\documentclass[letterpaper,english,cleveref,numberwithinsect]{lipics-v2021}
\usepackage[utf8]{inputenc}

\usepackage{csquotes}
\usepackage{cite}
\usepackage{todonotes}
\usepackage{mathtools}
\usepackage{tabulary}
\usepackage{multirow}
\usepackage{booktabs}
\usepackage{colortbl}

\crefformat{footnote}{#2\footnotemark[#1]#3}

\nolinenumbers

\graphicspath{{figures/}}

\newcommand{\A}{\raisebox{-1.5pt}{\includegraphics[scale=0.9,page=3]{figures/small-icons.pdf}}}
\newcommand{\B}{\raisebox{-1.5pt}{\includegraphics[page=4]{figures/small-icons.pdf}}}
\newcommand{\C}{\raisebox{-1.5pt}{\includegraphics[page=5]{figures/small-icons.pdf}}}
\newcommand{\D}{\raisebox{-2pt}{\includegraphics[scale=0.9,page=2]{figures/small-icons.pdf}}}
\newcommand{\E}{\raisebox{-2pt}{\includegraphics[page=6]{figures/small-icons.pdf}}}
\newcommand{\F}{\raisebox{-2pt}{\includegraphics[page=7]{figures/small-icons.pdf}}}
\newcommand{\T}{\raisebox{-3pt}{\includegraphics[page=8]{figures/small-icons.pdf}}}

\newcommand{\calC}{\ensuremath{\mathcal{C}}}
\newcommand{\calX}{\ensuremath{\mathcal{X}}}
\newcommand{\calS}{\ensuremath{\mathcal{S}}}

\newcommand{\calB}{\ensuremath{\mathcal{B}}}
\newcommand{\calT}{\ensuremath{\mathcal{T}}}

 \newcommand{\Oh}{\mathcal{O}}

\title{The Density Formula:\texorpdfstring{\\
One Lemma to Bound Them All}{}}

\titlerunning{The Density Formula}

\author{Michael Kaufmann}{Universit\"at T\"ubingen}{michael.kaufmann@uni-tuebingen.de}{https://orcid.org/0000-0001-9186-3538}{}
\author{Boris Klemz}{Universit\"at W\"urzburg}{firstname +-+- dot +-+- lastname +-+- at +-+- uni +-+- minus +-+- wuerzburg +-+- dot +-+- de}{https://orcid.org/0000-0002-4532-3765}{}
\author{Kristin Knorr}{Freie Universit\"at Berlin}{knorrkri@inf.fu-berlin.de}{https://orcid.org/0000-0003-4239-424X}{}
\author{Meghana M.\ Reddy}{ETH Z\"urich}{meghana.mreddy@inf.ethz.ch}{https://orcid.org/0000-0001-9185-1246}{}
\author{Felix Schr\"oder}{Technische Universit\"at Berlin, Charles University Prague}{schroder@kam.mff.cuni.cz}{https://orcid.org/0000-0001-8563-3517}{supported by the Czech Science Foundation grant GA\v CR 23-04949X}
\author{Torsten Ueckerdt}{Karlsruhe Institute of Technology}{torsten.ueckerdt@kit.edu}{https://orcid.org/0000-0002-0645-9715}{supported by the Deutsche Forschungsgemeinschaft - 520723789}
\authorrunning{M. Kaufmann, B. Klemz, K. Knorr, M.\,M. Reddy, F. Schr\"oder, and T. Ueckerdt}

\Copyright{Michael Kaufmann, Boris Klemz, Kristin Knorr, Meghana M.\ Reddy, Felix Schr\"oder, and Torsten Ueckerdt}



\EventEditors{Stefan Felsner and Karsten Klein}
\EventNoEds{2}
\EventLongTitle{32nd International Symposium on Graph Drawing and Network Visualization (GD 2024)}
\EventShortTitle{GD 2024}
\EventAcronym{GD}
\EventYear{2024}
\EventDate{September 18--20, 2024}
\EventLocation{Vienna, Austria}
\EventLogo{}
\SeriesVolume{320}
\ArticleNo{5}

\begin{document}

\maketitle

\begin{abstract}
    We introduce the Density Formula for (topological) drawings of graphs in the plane or on the sphere, which relates the number of edges, vertices, crossings, and sizes of cells in the drawing.
    We demonstrate its capability by providing several applications:
    we prove tight upper bounds on the edge density of various beyond-planar graph classes, including so-called $k$-planar graphs with $k=1,2$, fan-crossing / fan-planar graphs, $k$-bend RAC-graphs with $k=0,1,2$, quasiplanar graphs, and $k^+$-real face graphs.
    In some cases ($1$-bend and $2$-bend RAC-graphs and fan-crossing / fan-planar graphs), we thereby obtain the first tight upper bounds on the edge density of the respective graph classes.
    In other cases, we give new streamlined and significantly shorter proofs for bounds that were already known in the literature.
    Thanks to the Density Formula, all of our proofs are mostly elementary counting and mostly circumvent the typical intricate case analysis found in earlier proofs.
    Further, in some cases (simple and non-homotopic quasiplanar graphs), our alternative proofs using the Density Formula lead to the first tight lower bound examples.
\end{abstract}

\newpage 

\section{Introduction}
\label{sec:introduction}

Topological Graph Theory is concerned with the analysis of graphs drawn in the plane $\mathbb{R}^2$ or the sphere $\mathbb S^2$ such that the drawing has a certain property often related to forbidden crossing configurations.
The most prominent example is the class of planar graphs, which admit drawings without any crossings.
Other well-studied examples include $k$-planar graphs where every edge can have up to $k$ crossings, RAC-graphs where edges are straight line segments and every crossing happens at a right angle, or quasiplanar graphs where no three edges are allowed to pairwise cross each other.
As all these include planar graphs as a special case, they are commonly known as \emph{beyond-planar} graph classes.
See~\cite{DLM19-survey} for a recent survey.

When studying a beyond-planar graph class $\mathcal{G}$,
one of the most natural and important questions is to determine how many edges a graph in $\mathcal{G}$ can have.
The \emph{edge density} of $\mathcal{G}$ is the function giving the maximum number of edges over all $n$-vertex graphs in $\mathcal{G}$.
For example, planar graphs with at least three vertices have edge density $3n-6$.
All the beyond-planar graph classes mentioned above have linear edge density, i.e., their edge density is in $\Oh(n)$.
Proofs of precise linear upper bounds for the edge density of a specific class $\mathcal{G}$ are often times involved and very tailored to the specific drawing style that defines $\mathcal{G}$.
In particular, getting a \emph{tight} bound (even only up to an additive constant) was achieved only in a couple of cases.
A particularly simple case is the class of planar graphs, whose edge density of $3n-6$ can be easily derived from Euler's Formula.
However, a comparable formula for general drawings (with crossings) that can be used to easily derive tight upper bounds for the edge density of beyond-planar graph classes was not known --- until now.

\subparagraph{Our Contribution.}
In this paper, we introduce a new tool, called the
\textbf{Density Formula} (\cref{lem:density-formula}), which can
be used to derive upper bounds on edge densities for many beyond-planar graph classes.
It is an equation that relates the number of edges, vertices,
crossings, and sizes\footnote{\label{footnote:cells}
Loosely speaking, a cell of
a drawing is a connected region of the plane (or sphere) after
removing the drawing; its size is the number of vertex and
edge segment occurrences along its boundary, see \cref{fig:example-cells} for examples.} of cells\cref{footnote:cells} in a connected drawing of
a graph
and is parameterized by a real-valued parameter~$t$.
Intuitively, the Density Formula allows us to obtain density
bounds by counting the cells of small size in a drawing, which is often times quite an elementary
task.
The parameter~$t$ is chosen in accordance with the desired density bound, e.g., when aiming for a bound of roughly $5n$, where $n$ is the number of vertices, we might set $t=5$, in which case the Density Formula states that the
number of edges in the drawing is
$5n-10- \sum_{c \in \calC} (\|c\|-5) - x$,
where $x$ is the number of
crossings, $\mathcal C$ is the set of cells, and $\|c\|$ denotes
the size of a cell~$c$. 
Thus, any upper bound on $-(\sum_{c \in \calC} (\|c\|-5) + x)$
yields an upper bound on the number of edges.
Since the quantity $(\|c\|-5)$ is non-negative for cells of
size at least $5$, such a bound can indeed be obtained by counting the
cells of small sizes (here, at most $4$) and cross-charging them with the
crossings.

We give the precise, more general, statement of the Density
Formula in \cref{sec:all-drawings}, where we also develop some
general tools that help with the required counting / charging
arguments.
Before that, in \cref{sec:prelim}, we formally define some basic
notions, such as (connected) drawings, cells, cell sizes, etc., and discuss
some further preliminaries.
We demonstrate the capabilities of the Density Formula by
providing several applications, which are discussed next.

\begin{table}[!htb]
    \centering
    \begin{tabulary}{\textwidth}{CCCCC}
        \toprule
        beyond-planar & \multirow{2}{*}{variant} & \multicolumn{2}{c}{\multirow{2}{*}{upper bound}} & \multirow{2}{*}{lower bound} \\
        graph class & & & & \\
        \midrule
        \midrule
        \multirow{2}{*}{$0$-bend RAC} & \multirow{2}{*}{no constraint} & $4n-10$ & $4n-8$ & $4n-10$ \\
        & & \cite{DEL11-RAC} & \cref{thm:0-bend-RAC} & \cite{DEL11-RAC} \\
        \midrule
        \multirow{2}{*}{$1$-bend RAC} & \multirow{2}{*}{non-homotopic} & $5.4n-10.8$ & \cellcolor{lipicsYellow} $5n-10$ & $5n-10$ \\
        & & \cite{AnBeFoKa20} & \cellcolor{lipicsYellow} \cref{thm:1/2-bend-RAC} & \cite{AnBeFoKa20} \\
        \midrule
        \multirow{2}{*}{$2$-bend RAC} & \multirow{2}{*}{non-homotopic} & $20n-24$ & \cellcolor{lipicsYellow} $10n-19$ & \cellcolor{lipicsYellow} $10n-54$ \\
        & & \cite{CToth23} & \cellcolor{lipicsYellow} \cref{thm:1/2-bend-RAC} & \cellcolor{lipicsYellow} \cref{thm:2-bend-RAC-LB} \\
        \midrule
        fan-crossing / & \multirow{2}{*}{simple} & \textcolor{red!40}{$5n-10$} & \cellcolor{lipicsYellow} $5n-10$ & $5n-10$ \\
        fan-planar & & \textcolor{red!40}{\cite{KU14-arxiv,KaufUeck22,Brandenburg20,CFKPS23}} & \cellcolor{lipicsYellow} \cref{thm:fan-crossing} & \cite{KU14-arxiv,KaufUeck22} \\
        \midrule
        fan-cr. / fan-pl. & \multirow{2}{*}{simple} & \textcolor{red!40}{$4n-12$} & \cellcolor{lipicsYellow} $4n-10$ & $4n-16$ \\
        $+$ bipartite & & \textcolor{red!40}{\cite{ABKPU18,CFKPS23}} & \cellcolor{lipicsYellow} \cref{thm:bip-fan-crossing} & \cite{ABKPU18} \\
        \midrule
        \multirow{4}{*}{quasiplanar} & \multirow{2}{*}{simple} & $6.5n-20$ & $6.5n-20$ & \cellcolor{lipicsYellow} $6.5n-20$ \\
        & & \cite{AckeTard07} & \cref{thm:simple-quasiplanar} & \cellcolor{lipicsYellow} \cref{thm:LB-simple-quasiplanar} \\
        \cmidrule{2-5}
        & \multirow{2}{*}{non-homotopic} & $8n-20$ & $8n-20$ & \cellcolor{lipicsYellow} $8n-20$ \\
        & & \cite{AckeTard07} & \cref{thm:general-quasiplanar} & \cellcolor{lipicsYellow} \cref{thm:LB-general-quasiplanar} \\
        \midrule
        \multirow{2}{*}{$1^+$-real face} & \multirow{2}{*}{non-homotopic} & $5n-10$ & $5n-10$ & $5n-10$ \\
        & & \cite{DBLP:conf/wg/BinucciBDHKLMT23} &\cref{thm:1-vertex-per-cell}& \cite{DBLP:conf/wg/BinucciBDHKLMT23} \\
        \midrule
        \multirow{2}{*}{$2^+$-real face} & \multirow{2}{*}{non-homotopic} & $4n-8$ & $4n-8$ & $4n-8$ \\
        & & \cite{DBLP:conf/wg/BinucciBDHKLMT23} & \cref{thm:2-vertex-per-cell} & \cite{DBLP:conf/wg/BinucciBDHKLMT23} \\
        \midrule
        $k^+$-real face & \multirow{2}{*}{no constraint} & $\frac{k}{k-2}(n-2)$ & $\frac{k}{k-2}(n-2)$ & $\frac{k}{k-2}(n-2)$ \\
        $k \geq 3$ & & \cite{DBLP:conf/wg/BinucciBDHKLMT23} & \cref{thm:k-vertex-per-cell} & \cite{DBLP:conf/wg/BinucciBDHKLMT23} \\
        \midrule
        \multirow{2}{*}{$1$-planar} & \multirow{2}{*}{non-homotopic} & $4n-8$ & $4n-8$ & $4n-8$ \\
        & & \cite{PT97-k-planar} & \cref{thm:1-planar} & \cite{PT97-k-planar} \\
        \midrule
        \multirow{2}{*}{$2$-planar} & \multirow{2}{*}{non-homotopic} & $5n-10$ & $5n-10$ & $5n-10$ \\
        & & \cite{PT97-k-planar} & \cref{thm:2-planar}& \cite{PT97-k-planar} \\
        \midrule
        \midrule
        & & previous work & Density Formula & \\
        \bottomrule
    \end{tabulary}    
    \caption{
        Overview of edge density bounds, i.e., the maximum number of edges in connected $n$-vertex ($n$ large enough) graphs in that graph class.
        In particular, the third column lists previous work on upper bounds and the fourth column lists the upper bounds we obtain using the Density Formula.
        \colorbox{lipicsYellow}{Previously unknown bounds} are highlighted with boxes.
        Results from the literature that are written in light red rely on \textcolor{red!40}{incomplete proofs} as they use an incorrect statement from \cite{KU14-arxiv,KaufUeck22}, as we discuss in more detail in \cref{sec:fan-planar-gap}.
    }
    \label{tab:overview}
\end{table}

\subsubsection*{Applications}
\subparagraph{$k$-Bend RAC-Drawings.}
For an integer $k \geq 0$, a drawing $\Gamma$ in the plane $\mathbb{R}^2$ of some graph~$G$ is \emph{$k$-bend RAC}, which stands for right-angle crossing, if every edge of $\Gamma$ is a polyline with at most $k$ bends and every crossing in $\Gamma$ happens at a right angle, and in this case $G$ is called a \emph{$k$-bend RAC-graph}.
The $k$-bend RAC-graphs were introduced by Didimo, Eades, and Liotta~\cite{DEL11-RAC}, who prove that $n$-vertex $0$-bend RAC-graphs have at most $4n-10$ edges (and this is tight), while every graph is a $3$-bend RAC-graph.
The best known upper bound for simple\footnote{\label{footnote:simple}Loosely speaking, in a simple drawing,
every pair of edges intersects in at most one point, thereby forbidding digons formed by segments of two edges. In non-homotopic drawings, such digons are allowed as long as both regions bounded by a digon contain at least one vertex or crossing.
} $1$-bend RAC-drawings is $5.4n - 10.8$~\cite{AnBeFoKa20}, while the lower bound is $5n-10$~\cite{AnBeFoKa20}.
By means of the Density Formula, we give an improved upper bound of $5n-10$ for the connected case, which is best-possible.
Very recently, T\'oth~\cite{CToth23} established an upper bound of $20n - 24$ for simple graphs admitting $2$-bend RAC-drawings, thereby improving the long standing previous best upper bound of $74.2n$ \cite{DBLP:journals/comgeo/ArikushiFKMT12}.
Using the Density Formula, we derive a significantly improved upper bound of $10n-19$ for simple drawings.
We also show that this bound is tight up to an additive constant
by constructing an infinite family of simple  $2$-bend RAC-drawings with $10n-54$ edges.
(A similiar construction of $2$-bend RAC-drawings with $10n-46$ edges was presented by Angelini et al.~\cite{ABKKPU23}, but their drawings are not simple.)
Both of our upper bound results in fact apply even to the non-homotopic\cref{footnote:simple} case.
We prove these results in \cref{sec:k-bend-RAC-graphs}.
For completeness, in \cref{sec:k-bend-RAC-graphs-appendix} of the appendix, we also apply the Density Formula to reprove the known upper bound for 0-bend RAC-graphs.

\subparagraph{Fan-Crossing Drawings.}
A drawing $\Gamma$ on the sphere $\mathbb{S}^2$ of some graph $G$ is \emph{fan-crossing} if for every edge $e$ of $G$, the edges crossing $e$ in $\Gamma$ form a star in $G$, and in this case $G$ is called a \emph{fan-crossing graph}.
A simple drawing is fan-crossing if and only if there is no configuration~I and no triangle-crossing, as shown in \cref{fig:fan-planar-mess}.
Fan-crossing drawings generalize fan-planar drawings; but the story about fan-planar graphs is problematic and tricky.
In a preprint from 2014, Kaufmann and Ueckerdt~\cite{KU14-arxiv} introduced fan-planar drawings as the simple drawings in~$\mathbb{R}^2$ without configuration~I and~II, as shown in \cref{fig:fan-planar-mess}.
These are today known as \emph{weakly fan-planar} and they show that $n$-vertex weakly fan-planar graphs have at most $5n-10$ edges~\cite{KU14-arxiv}.
However, recently, a first flaw in this proof was discovered~\cite{KKRS23}.
It was fixed in the journal version~\cite{KaufUeck22} of~\cite{KU14-arxiv} from 2022 by additionally forbidding configuration~III, as shown in \cref{fig:fan-planar-mess}.
These, more restricted, graphs are today known as \emph{strongly fan-planar} graphs, 
and it is known that this indeed is a different graph class~\cite{CFKPS23}.
However, for each $n$-vertex weakly fan-planar graph, there is a strongly fan-planar graph on the same number of vertices and edges~\cite{CFKPS23}.
So any density result could be lifted.
As every triangle-crossing contains configuration~II, weakly fan-planar graphs are also fan-crossing, while again these are indeed different graph classes~\cite{Brandenburg20}.
However again, for each $n$-vertex fan-crossing graph, there is a weakly fan-planar graph on the same number of vertices and edges~\cite{Brandenburg20}, and thus the density can be lifted to fan-crossing graphs.

\begin{figure}[htb]
    \centering
    \includegraphics{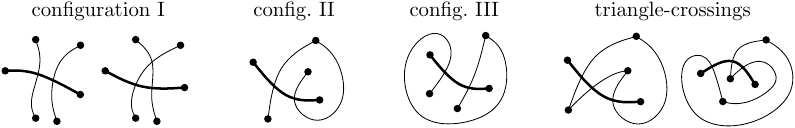}
    \caption{
        Simple fan-planar drawings have neither configuration~I, nor II, nor III.
        Simple fan-crossing drawings have no configuration~I and no triangle-crossings.
    }
    \label{fig:fan-planar-mess}
\end{figure}

In \cref{sec:fan-crossing}, we prove an upper bound of $5n-10$ for simple $n$-vertex connected fan-crossing drawings by applying the Density Formula.
We also briefly describe in \cref{sec:fan-planar-gap} another issue in the (updated) proof from~\cite{KaufUeck22} by providing a counterexample to one of their crucial statements.
As all previous density results rest on~\cite{KaufUeck22}, our result on fan-crossing drawings is the first complete proof for fan-crossing, weakly fan-planar, and strongly fan-planar drawings.
Moreover, our proof is significantly simpler than the strategy used in~\cite{KaufUeck22}.
In \cref{sec:bipartite-fan-crossing} of the appendix, we discuss the special case of fan-crossing drawings of bipartite graphs, for which we obtain similar results.

We remark that, in a very recent preprint, Ackerman and Keszegh~\cite{AckKes23} also (independently of us) propose a new alternative proof for the $5n-10$ upper bound for fan-crossing graphs.
Moreover, Brandenburg~\cite{Brandenburg20} also considers \emph{adjacency-crossing} graphs by just forbidding configuration~I, but allowing triangle-crossings.
He shows however that this class coincides with fan-crossing graphs, and hence our $5n-10$ upper bound applies.

\subparagraph{Quasiplanar Drawings.}
A drawing $\Gamma$ on the sphere $\mathbb{S}^2$ of some graph $G$ is \emph{quasiplanar} if no three edges of $G$ pairwise cross in $\Gamma$ and in this case $G$ is called a \emph{quasiplanar graph}.
Quasiplanar graphs were introduced by Pach~\cite{Pac91}.
It is known that simple $n$-vertex quasiplanar drawings have at most $6.5n - 20$ edges~\cite{AckeTard07} and non-homotopic connected $n$-vertex quasiplanar drawings have at most $8n-20$ edges~\cite{AckeTard07}.
However, the best known lower bounds~\cite{AckeTard07} are just $6.5n - 29$ and $7n - 29$, respectively.
In \cref{sec:quasiplanar-appendix} and \cref{sec:simple-quasiplanar-appendix} of the appendix, we reprove the known upper bounds using the Density Formula.
In \cref{sec:quasiplanar}, inspired by insights gained in our upper bound proofs, we provide families of drawings showing that the previous upper bounds are actually best possible.

\subparagraph{Further applications.}
In \cref{sec:k-real-face-graphs} and \cref{sec:k-planar} of the appendix, we also use the Density Formula to reprove (and slightly generalize) the known upper bounds for so-called $k^+$-real face graphs and 1-planar and 2-planar graphs respectively. 
All results are summarized in \cref{tab:overview}.

Some previously known density proofs already contain ideas that are similar to (parts of) our strategy, and we discuss this further in \cref{sec:conclusions}. 
But with the Density Formula, whose proof is merely a straight-forward application of Euler's Formula, we have a unified and simple approach that somewhat unveiled the essential tasks in this field of research.
We believe it will serve as a useful tool for proving density bounds in the future.
For example, very recently, the Density Formula was already applied by Bekos et al.~\cite{bekos24-k-planar} to give bounds on the density of $k$-planar graphs without short cycles.
Moreover, given that the Density Formula behaves symmetrically when it comes to the number of edges and the number of crossings, it seems plausible that it can also be used to derive bounds on crossing numbers of beyond-planar graph classes.


\section{Terminology, Conventions, and Notation}
\label{sec:prelim}

All graphs in this paper are finite and have no loops, but possibly parallel edges.
We consider classic node-link drawings of graphs.
More precisely, in a \emph{drawing} $\Gamma$ of a graph $G = (V,E)$ (in the plane $\mathbb R^2$ or on the sphere $\mathbb S^2$) the vertices are pairwise distinct points and each edge is a simple\footnote{\label{footnote:simple-Jordan}with no self-intersection} Jordan curve connecting the two 
vertices.
In particular, no edge crosses itself.
In order to avoid special treatment of the unbounded region, we mostly consider drawings on the sphere $\mathbb{S}^2$.
In one case (RAC-drawings), however, we consider drawings in the plane $\mathbb{R}^2$, as the drawing style involves straight lines and angles.
In any case, we require throughout the usual assumptions of no edge passing through a vertex, having only proper crossings and no touchings, only finitely many crossings, and no three edges crossing in the same point.

So-called \emph{simple} drawings are a particularly important and well-studied type of drawing.
In such a drawing, any two edges have at most one point in common.
In particular, simple drawings contain
no two edges crossing more than once,
no crossing adjacent edges, and
no parallel edges.
However, there is increasing interest in generalizations of simple drawings that allow these types of configurations, as long as the involved edge pairs are not just drawn in basically the same way within a narrow corridor.
This notion is formalized as follows.
A \emph{lens} in a drawing $\Gamma$ is a region whose boundary is described by a simple\cref{footnote:simple-Jordan} closed Jordan curve $\gamma$ such that $\gamma$ is comprised of exactly two contiguous parts, each being formed by (a part of) one edge.
So the curve $\gamma$ consists of either two non-crossing parallel edges, or parts of two crossing adjacent edges, or parts of two edges crossing more than once;
see \cref{fig:bad-lenses} for illustrations.
Be aware that for drawings in $\mathbb{R}^2$, a lens might be an unbounded region.
Now let us call a drawing $\Gamma$ \emph{non-homotopic}\footnote{Usually, non-homotopic drawings require a vertex in each lens, but we only need our weaker requirement.} if every lens contains a vertex or a crossing in its interior.
This is indeed a generalization of simple drawings, as these cannot contain any lens.

\begin{figure}[htb]
    \centering
    \includegraphics{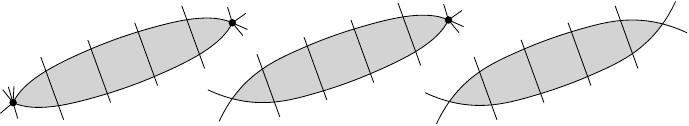}
    \caption{Lenses with no vertex and no crossing in their interior. Such configurations are forbidden in non-homotopic drawings.}
    \label{fig:bad-lenses}
\end{figure}

Beyond-planar graph classes are implicitly defined as all graphs $G$ that admit a drawing~$\Gamma$ with specific properties, such as all edges of $G$ having at most one crossing in $\Gamma$.
These for example are called \emph{$1$-planar drawings}\footnote{In literature, planar drawings are also referred to as \emph{plane drawings}, and a planar graph with a fixed plane drawing is called a \emph{plane graph}.
And there is a similar distinction for each beyond-planar graph class (e.g., $1$-planar vs.~$1$-plane graphs).
But for simplicity, we treat \emph{planar} and \emph{plane} as equivalent here.} 
and the corresponding graphs are called \emph{$1$-planar graphs}.
We extend this policy to the properties `simple' and `non-homotopic' in the same way, e.g., a non-homotopic $1$-planar graph is a graph that admits a non-homotopic $1$-planar drawing.
Observe that this aligns with a simple graph being a graph with no loops (which we rule out entirely) and no parallel edges.

\smallskip

Fix a drawing $\Gamma$ of some graph $G=(V,E)$.
Setting up some notation, let $E_x \subseteq E$ be the set of all \emph{crossed} edges of $G$, i.e., with at least one crossing in $\Gamma$, and $E_p = E\setminus E_x$ be the set of all \emph{planar} edges (without crossings).
Further, let $\calX$ denote the set of all crossings in~$\Gamma$.
Each edge $e$ is split into one or more \emph{edge-segments} by the crossings along $e$.
That is, an edge with exactly $k$ crossings, $k \geq 0$, is split into exactly $k+1$ edge-segments.
An \emph{outer} edge-segment of $\Gamma$ is incident to some vertex, while an \emph{inner} edge-segment is not.
The set of all edge-segments of $\Gamma$ is denoted by $\calS$ and the set of all inner edge-segments by $\calS_{\rm in}$.

\begin{observation}\label{obs:inner-segments}
    Let $\Gamma$ be any drawing of some graph $G = (V,E)$.
    Then
    \[
        |\calS| = 2|\calX| + |E| \quad \text{and} \quad |\calS_{\rm in}| = |\calS| - 2 |E_x| - |E_p| = 2|\calX| - |E_x|.
    \]
\end{observation}
%

The \emph{planarization} $\Lambda$ of the drawing $\Gamma$ is the planar drawing obtained from $\Gamma$ by replacing each crossing by a new vertex and replacing each edge by its edge-segments.
We call the drawing $\Gamma$ \emph{connected} if the graph underlying its planarization $\Lambda$ is connected.
Let us remark that most density results in this paper assume for brevity the considered graphs to be connected, while our proofs actually only require the respective drawings to be connected.

The connected components of $\mathbb{S}^2$ or $\mathbb{R}^2$ after removing all edges and vertices in $\Gamma$ are called the \emph{cells} of $\Gamma$.
The set of all cells is denoted by $\calC$. 
The \emph{boundary} $\partial c$ of each cell $c$ consists of a cyclic sequence alternating between $V \cup \calX$ and $\calS$, i.e., vertices/crossings and edge-segments of $\Gamma$.
If $\Gamma$ is not connected, $\partial c$ might consist of multiple such sequences.
Be aware that an edge-segment might appear twice on $\partial c$, a crossing might appear up to four times on $\partial c$, and a vertex $v$ may appear up to $\deg(v)$ times on $\partial c$.
Each appearance of an edge-segment / vertex / crossing on $\partial c$ is called an edge-segment-incidence / vertex-incidence / crossing-incidence of $c$.
The total number of edge-segment-incidences and vertex-incidences of $c$ is called the \emph{size} of $c$ and denoted by $\|c\|$.
Note that $\|c\|$ does not take the number of crossings on $\partial c$ into account, while edge-segments and vertices are counted with multiplicities; see \cref{fig:example-cells} for several examples. 
For an integer $i$, let $\calC_i$ and $\calC_{\geq i}$ denote the set of all cells of size exactly $i$ and the set of all cells of size at least $i$, respectively.

\begin{figure}[htb]
    \centering
    \includegraphics{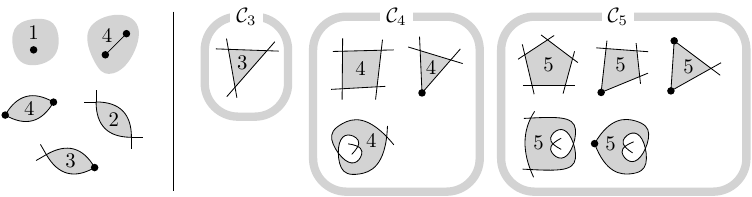}
    \caption{
        Left: Cells (and their sizes) that do not appear in simple or non-homotopic drawings on at least three vertices.
        Right: All types of cells $c$ of size $\|c\| \leq 5$ in a non-homotopic connected drawing on at least three vertices (cf.\ \cref{obs:types-of-cells}).
        The bottom row shows degenerate \B-cells, \F-cells, and \E-cells.
    }
    \label{fig:example-cells}
\end{figure}

\cref{fig:example-cells}(right) shows all possible types of cells of size at most~$5$ that can occur in connected non-homotopic drawings on $\mathbb{S}^2$ with at least three vertices; 
the bottom row shows the \emph{degenerate} cells, i.e., those cells $c$ with a crossing or vertex appearing repeated in $\partial c$. 
When proving edge density bounds by means of the Density Formula, our main task will be to count these ``small cells''; we denote the different types for convenience with little pictograms, such as \A-cells, \B-cells, and \D-cells.
More precisely, each pictogram describes a type of cell~$c$ in terms of the sequence of types of incidences found along its connected boundary~$\partial c$.
E.g., the boundary of a \A-cell consists of a vertex-incidence, an edge-segment-incidence, a crossing-incidence, an edge-segment-incidence, a crossing-incidence, and an edge-segment-incidence, in this order.

\begin{observation}\label{obs:types-of-cells}
    Let $\Gamma$ be any non-homotopic connected drawing of some graph $G$ with at least three vertices.
    Then
    \begin{itemize}
        \item $\calC_3$ is the set of all \C-cells,
        \item $\calC_4$ is the set of all \A-cells and \B-cells, and
        \item $\calC_5$ is the set of all \D-cells, \E-cells, and \F-cells.\qed
    \end{itemize}
\end{observation}

\section{The Density Formula}
\label{sec:all-drawings}

In this section, we first state and prove the Density Formula, then derive some immediate consequences, and finally develop some general tools that are useful for its application.

\begin{lemma}[Density Formula]\label{lem:density-formula}
    Let $t$ be a real number.
    Let $\Gamma$ be a connected drawing of a graph $G = (V,E)$ with at least one edge.
    Then
    \begin{equation*}
        |E| = t(|V|-2)- \sum_{c \in \calC} \left(\frac{t-1}{4}\|c\|-t\right) - |\calX|.
    \end{equation*}
\end{lemma}
\begin{proof}
    First recall that, by \cref{obs:inner-segments},
    \begin{equation}
        |\calS| = |E| + 2|\calX|.\label{eq:segment-count}
    \end{equation}
    
    Considering the total sum of $\|c\|$ over all cells $c \in \calC$, we count every vertex $v \in V$ exactly $\deg(v)$ times and every edge-segment exactly twice.
    (Here we use that $|E| \geq 1$ and, thus, $\deg(v) \geq 1$ for each $v \in V$ as $\Gamma$ is connected.)
    Thus,
    \begin{equation}
        \sum_{c \in \calC} \frac14 \|c\| = \frac14 \left( \sum_{v \in V} \deg(v) + 2|\calS|\right) = \frac14 (2|E| + 2|\calS|) \overset{\eqref{eq:segment-count}}{=}\frac14(4|E| + 4|\calX|) = |E| + |\calX|. \label{eq:density-proof-1}
    \end{equation}
    Let $\Lambda = (V_\Lambda,E_\Lambda)$ be the planarization of $G$. It has exactly $|V_\Lambda|=|V| + |\calX|$ vertices, $|E_\Lambda|=|\calS|$ edges, and $|\calC|$ faces.
    As $\Lambda$ is connected, we can apply Euler's Formula ($\ast$):
    \[
        |E| + 2|\calX| \overset{\eqref{eq:segment-count}}{=} |\calS| = |E_\Lambda| \overset{\text{($\ast$)}}{=} |V_\Lambda| + |\calC| - 2 = |V| - 2 + |\calC| + |\calX|,
    \]
    which gives the following two equations:
    \begin{eqnarray}
        |E| =& |V|-2 + |\calC| - |\calX| &= \hspace{0.2em} (|V|-2) - \hspace{0.5em} \sum_{c \in \calC}\left( -1 \right) \hspace{3em} - |\calX| \label{eq:density-proof-2}\\
        0 =& |V|-2 - (|E|+|\calX|) + |\calC| &\overset{\eqref{eq:density-proof-1}}{=} (|V|-2) - \sum_{c \in \calC} \left( \frac14 \|c\| - 1\right)\label{eq:density-proof-3}
    \end{eqnarray}
    Adding \eqref{eq:density-proof-2} and $(t-1)$ times \eqref{eq:density-proof-3} gives the result.
    %
\end{proof}

The Density Formula can be used to find upper bounds on edge densities by counting cells of small size.
To see how this works, let us plug in two specific values for~$t$ ($t=4$ and $t=5$, which we use quite often throughout the paper) resulting in the following statements:

\begin{corollary}\label{cor:density-formula-4}
    For any connected drawing $\Gamma$ of a graph $G = (V,E)$ with $|E| \geq 1$ we have
    \begin{equation*}
        |E| = 4|V|-8- \sum_{c \in \calC} (\frac34\|c\|-4) - |\calX| \leq 4|V|-8+\frac74|\calC_3| + |\calC_4| + \frac14|\calC_5| - |\calX|.
    \end{equation*}
\end{corollary}

\begin{corollary}\label{cor:density-formula-5}
    For any connected drawing $\Gamma$ of a graph $G = (V,E)$ with $|E| \geq 1$ we have
    \begin{equation*}
        |E| = 5|V|-10- \sum_{c \in \calC} (\|c\|-5) - |\calX| = 5|V|-10+2|\calC_3|+|\calC_4|-|\calX|-\sum_{c \in \calC_{\geq 5}}(\|c\|-5).
    \end{equation*}
\end{corollary}

So indeed, \cref{cor:density-formula-4} allows us to derive upper bounds on $|E|$ by proving upper bounds on $\frac74|\calC_3| + |\calC_4| + \frac14|\calC_5| - |\calX|$, which can be done by counting cells of sizes $3$, $4$, and $5$ and cross-charging them with the crossings.
Similarly, noting that $\sum_{c \in \calC_{\geq 5}} (\|c\|-5)$ is non-negative, \cref{cor:density-formula-5} allows us to derive upper bounds on $|E|$ by proving upper bounds on $2|\calC_3| + |\calC_4| - |\calX|$.
In fact, by taking into account the cells of larger sizes, one can sometimes obtain more precise bounds.
Thus, in the remainder of the section, we will devise some general tools that help with the required counting / charging arguments.
Moreover, we give a first concrete example of such an argument by proving \cref{lem:4n-8-more-general}, which is a simple but very general statement --- in fact, it immediately gives two bounds of $4n-8$ in \cref{tab:overview}.

\begin{lemma}\label{lem:4n-8-more-general}
    Let $\Gamma$ be a non-homotopic connected drawing of a graph $G = (V,E)$ with $|V| \geq 3$ and with no \C-cells, no \B-cells, no \F-cells, no \A-cells, and no \E-cells.
    Then $|E|\leq 4|V|-8$.
\end{lemma}
\begin{proof}
    By assumption and \cref{obs:types-of-cells}, we have $|\calC_3| = 0$ and $|\calC_4| = 0$ and $|\calC_5| = \#$\D-cells.
    Clearly, every crossing is incident to at most four \D-cells and every \D-cell has one incident crossing.
    In particular, it follows that $\#$\D-cells $\leq 4|\calX|$.
    Therefore, the Density Formula with $t = 4$ (\cref{cor:density-formula-4}) immediately gives
    \[
        \hspace*{-1em}|E| \leq 4|V|-8 + \frac74|\calC_3| + |\calC_4| + \frac14|\calC_5| - |\calX| = 4|V|-8 +\frac14\#\text{\D-cells} - |\calX| \leq 4|V| - 8.\qedhere
    \]
\end{proof}

\begin{lemma}\label{lem:A-cells-vs-X}
    Let $\Gamma$ be any non-homotopic drawing.
    Then $\#\text{\A-cells} \leq |\calX|$.
    Moreover, one can assign each \A-cell $c$ a crossing in $\partial c$ such that each crossing is assigned at most once.
\end{lemma}
\begin{proof}
    At every crossing incident to a \A-cell there is one inner edge-segment and one outer edge-segment.
    As $\Gamma$ is non-homotopic, every inner edge-segment is incident to at most one \A-cell.
    This implies that every crossing is incident to at most two \A-cells, while every \A-cell has two distinct incident crossings, which implies the claim.
\end{proof}

\begin{lemma}\label{lem:B-general}
    Let $\Gamma$ be a connected non-homotopic drawing of some graph $G$ with at least three vertices.
    Then
    \[
        |\calS_{\rm in}| \geq \#\text{\A-cells} + 2\cdot\#\text{\B-cells} + 3\cdot\#\text{\C-cells} \quad \text{and} \quad |\calS_{\rm in}| + \#\text{\A-cells} \geq 2|\calC_4| + 3|\calC_3|.
    \]
\end{lemma}
\begin{proof}
    The second inequality follows by combining the first inequality with \cref{obs:types-of-cells}.
    To prove the first inequality,
    let us call an inner edge-segment \emph{bad} if it is incident to a \A-cell or \C-cell in $\Gamma$.
    As~$\Gamma$ is non-homotopic, every bad edge-segment is incident to only one \A-cell or \C-cell.
    Hence, for the set $\calB \subseteq \calS_{\rm in}$ of all bad edge-segments we have $|\calB| = \#\text{\A-cells} + 3\cdot\#\text{\C-cells}$.
    Define an auxiliary graph $J = (V_J,E_J)$ with vertex set $V_J = \calS_{\rm in}$ and with two edge-segments being adjacent in $J$ if and only if they are an opposite pair of edge-segments for some \B-cell.
    Note that this and the following is true whether the \B-cells are degenerate or not.
    Then $|V_J| = |\calS_{\rm in}|$ and $|E_J| = 2\cdot \#\text{\B-cells}$, and the maximum degree in $J$ is at most two.
    Observe that $J$ contains no cycle, as such a cycle would correspond to a cyclic arrangement of \B-cells and therefore two edges in $G$ with no endpoints.
    Hence, $J$ is a disjoint union of paths (possibly of length $0$) and every bad edge-segment is an endpoint of one such path.
    Further, no path in $J$ on two or more vertices can have two bad endpoints, as such a path would correspond to a lens in $\Gamma$ containing no vertex and no crossing (as illustrated in \cref{fig:bad-lenses}), contradicting the fact that~$\Gamma$ is non-homotopic.
    Note that this implies $|V_J|\geq |E_J|+|\calB|$.
    Recalling that $|\calB| = \#\text{\A-cells} + 3\cdot\#\text{\C-cells}$, $|V_J| = |\calS_{\rm in}|$ and $|E_J| = 2\cdot \#\text{\B-cells}$, we obtain the first inequality of the lemma.
\end{proof}

\section{\texorpdfstring{$\boldsymbol{k}$}{k}-Bend RAC-Graphs}
\label{sec:k-bend-RAC-graphs}

In this section, we present our results for 1-bend and 2-bend RAC-graphs. We begin with the upper bounds, for which we only require the following lemma.

\begin{lemma}\label{lem:1/2-bend-RAC}
    Let $k\in\{1,2\}$ and $\Gamma$ be a non-homotopic drawing of a connected graph $G = (V,E)$ such that every crossed edge $e \in E_x$ is a polyline with at most $k$ bends, and every crossing is a right-angle crossing.
    Then $2|\calC_3| + |\calC_4| \leq |\calX| +\frac{k-1}{2}(|E_x|+1)$.
\end{lemma}
\begin{proof}
    %

    \cref{lem:B-general} gives
    \begin{equation}
        |\calS_{\rm in}| \geq \#\text{\A-cells} + 2 \cdot \#\text{\B-cells} + 3 \cdot \#\text{\C-cells}.\label{eq:inner-ineq-1}
    \end{equation}
    Now, each \A-cell and each \C-cell $c$ is a polygon, and as all crossings have right angles, $c$ has at least one convex corner that is a bend, except when $c$ is the unbounded cell.
    As every bend is a convex corner for only one cell, we have
    \begin{equation}
        k|E_x| \geq  \#\text{\A-cells} + \#\text{\C-cells} - 1.\label{eq:inner-ineq-2}
    \end{equation}
    Together this gives the desired
    \[
        4|\calC_3| + 2|\calC_4| \overset{\eqref{eq:inner-ineq-1},\eqref{eq:inner-ineq-2}}{\leq} |\calS_{\rm in}| + k|E_x| + 1 = 2|\calX| + (k-1)|E_x| + 1,
    \]
    where the last equality uses $|\calS_{\rm in}| = 2|\calX|-|E_x|$ from \cref{obs:inner-segments}.
    Dividing by $2$ and realizing that $2|\calC_3| + |\calC_4|$ is an integer, concludes the proof.
\end{proof}

\begin{theorem}\label{thm:1/2-bend-RAC}
    For every $k \in \{1,2\}$ and every $n \geq 3$, every connected non-homotopic $n$-vertex $k$-bend RAC-graph $G$ has at most $k(5n-10)+(k-1)$ edges.
\end{theorem}
\begin{proof}
    Let $\Gamma$ be a non-homotopic $k$-bend RAC-drawing of $G = (V,E)$.
    As $G$ is connected, so is $\Gamma$.
    The Density Formula with $t = 5$ (\cref{cor:density-formula-5}) and \cref{lem:1/2-bend-RAC} immediately give
    \[
        |E| \leq 5|V|-10 + 2|\calC_3| + |\calC_4| - |\calX| \leq 5|V|-10 + \frac{k-1}{2}(|E_x|+1),
    \]
    which implies the desired $|E| \leq |E| + (k-1)|E_p| \leq k(5|V|-10)+(k-1)$.
\end{proof}



The lower bound construction in~\cite[Theorem 6]{ABKKPU23} gives $2$-bend RAC-graphs with $n$ vertices and $10n-46$ edges, but the provided drawings are not simple (not even non-homotopic).
We modify it giving simple $2$-bend RAC-graphs with $n$ vertices and $10n - 54$ edges.

\begin{figure}[htb]
    \centering
    \includegraphics[width=\textwidth]{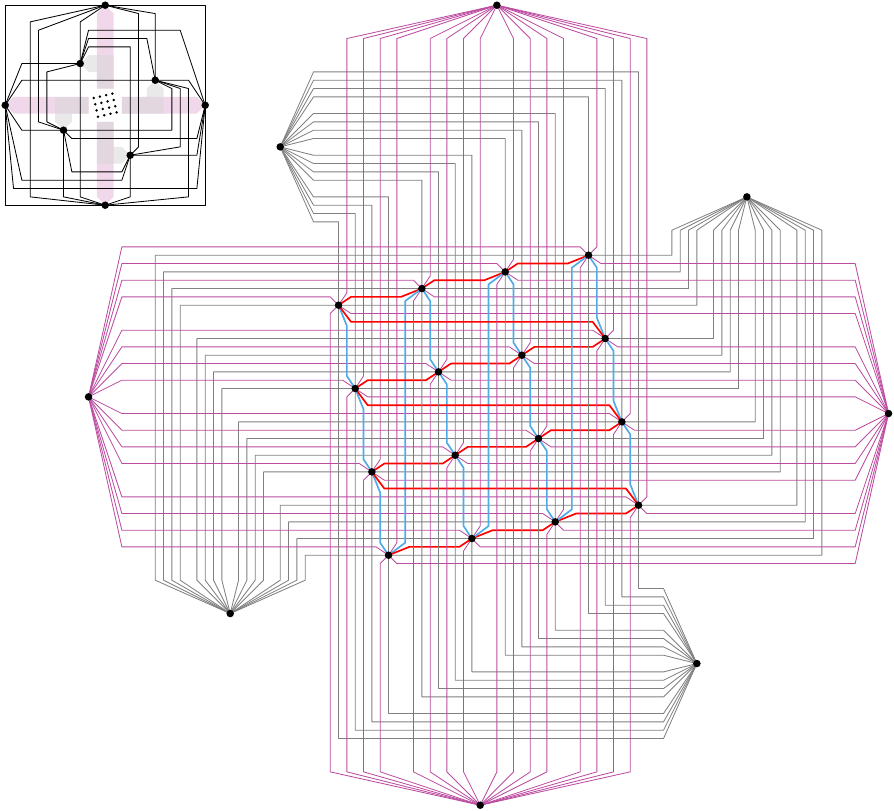}
    \caption{(Illustration of) a simple $2$-bend RAC-drawing of $G_4$ from \cref{thm:2-bend-RAC-LB}.}
    \label{fig:2-bend-ap-RAC}
\end{figure}

\begin{theorem}\label{thm:2-bend-RAC-LB}
    For every integer $k \geq 1$ there exists a simple connected $2$-bend RAC-graph $G_k$ with $n = k^2 + 8$ vertices and $10n-54$ edges.
\end{theorem}
\begin{proof}
    For $k \geq 1$, a simple 2-bend RAC-drawing of the graph $G_k$ (\cref{fig:2-bend-ap-RAC}) consists of
    \begin{itemize}
        \item a set $Q$ of $k^2 = n-8$ vertices in a regular but slightly rotated $k \times k$ grid,
        \item an $x$-monotone $2$-bend edge between any two vertices of $Q$\\
            with consecutive $y$-coordinates (red), \hfill ($n-9$ edges)
        \item a $y$-monotone $2$-bend edge between any two vertices of $Q$\\
            with consecutive $x$-coordinates (blue), \hfill ($n-9$ edges)
        \item a set $P$ of eight vertices around $Q$, each connected to all vertices\\of $Q$
            with either all (weakly) $x$-monotone $2$-bend edges or all\\(weakly) $y$-monotone $2$-bend edges (gray and purple), \hfill ($8(n-8)$ edges)
        \item a $2$-bend edge between any two vertices of $P$ (black). \hfill ($28$ edges)
    \end{itemize}
    The routing of the edges is illustrated in \cref{fig:2-bend-ap-RAC}.
\end{proof}

\section{Fan-Crossing Graphs}
\label{sec:fan-crossing}

Here, we present our upper bound for fan-crossing graphs, starting with the key lemma.

\begin{lemma}\label{lem:adj-crossing}
    Let $\Gamma$ be a simple connected fan-crossing drawing of a graph with at least three vertices.
    Then $|\calC_4| \leq |\calX|$.
\end{lemma}
\begin{proof}
    First, observe that there are no degenerate \B-cells since $\Gamma$ is simple.
    We shall map each cell $c \in \calC_4$ onto one of its incident crossings $\phi(c)$ in such a way that no crossing is used more than once, i.e., the mapping $\phi \colon \calC_4 \to \calX$ is injective.

    As an auxiliary structure, we orient edge-segments incident to \B-cells as follows.
    Let $c$ be a \B-cell and $s,s'$ be a pair of \emph{opposite} edge-segments in $\partial c$ (that do not share a crossing).
    As $\Gamma$ is simple, the corresponding edges $e,e'$ are distinct.
    Now orient $s$ and $s'$, each towards the (unique) common endpoint of $e$ and $e'$, which exists as $\Gamma$ is fan-crossing.
    Doing this for every \B-cell and every pair of opposite edge-segments, we obtain a well-defined  orientation:

    \begin{claim*}
        An edge-segment $s$ shared by two \B-cells $c_1,c_2$ has the same orientation in both.
    \end{claim*}
    \begin{claimproof}
    Observe that the six crossings incident to $c_1$ and $c_2$ are pairwise distinct since~$\Gamma$ is a simple drawing.
        Let $e = uv$ be the edge containing $s$ and $e_1,e_2$ be the two (distinct) edges crossing $e$ at the endpoints of $s$ (which are crossings in $\Gamma$).
        Further, let $f_1,f_2$ be the two edges containing the edge-segment opposite to $s$ in $c_1,c_2$, respectively.
        In particular, $e,f_1,f_2$ all cross $e_1$ and all cross $e_2$.
        As $\Gamma$ is fan-crossing\footnote{Here it is crucial that $e,f_1$ and $f_2$ do not form a triangle-crossing.}, $e,f_1,f_2$ have a common endpoint, say $u$.
        But then $s$ is oriented consistently towards $u$ according to both incident \B-cells $c_1,c_2$.
    \end{claimproof}

    \begin{claim*}
        For each \B-cell $c$, there is at least one crossing $x$ incident to~$ c$ such that both edge-segments incident to $c$ and $x$ are oriented outgoing from $x$.
    \end{claim*}
    \begin{claimproof}
        Assuming otherwise, the edge-segments would be oriented cyclically around $\partial c$.
        Consider two crossings $x_1,x_2$ that are \emph{opposite} along $c$ (do not belong to the same edge segment of $c$).
        The edges of the two (distinct) edge-segments of $c$ that are outgoing from $x_1,x_2$ have a common endpoint $u$, as $\Gamma$ is fan-crossing; see \cref{fig:cyclic-assignment}.
        The edges of the two edge-segments of $c$ that are outgoing from the remaining two opposite crossings $y_1, y_2$ behave symmetrically and share an endpoint $v$, which is distinct from $u$, as $\Gamma$ is simple.
        The four parts of the mentioned edges that join the vertices $u,v$ with the crossings $x_1,x_2,y_1,y_2$ are pairwise crossing-free since $\Gamma$ is simple.
        Hence, using these edge parts, we can obtain a planar drawing  of the bipartite graph $K_{3,3} - e$ (obtained from $K_{3,3}$ by removing an edge) so that the bipartition classes are $\{x_1,x_2,v\}$ and $\{y_1,y_2,u\}$ and where the four degree-3 vertices form a face.
        However, the unique\footnote{\label{footnote:K33}All planar embeddings of $K_{3,3}-e$ are combinatorially isomorphic since it is a subdivision of the $3$-connected complete graph $K_4$.} planar embedding of $K_{3,3}-e$ has no such face; see again \cref{fig:cyclic-assignment}.
    \end{claimproof}

    \begin{figure}[htb]
        \centering
        \includegraphics[page=2]{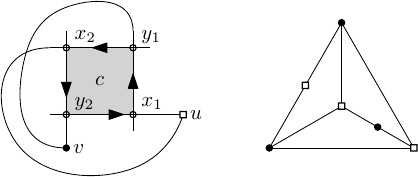}
        \caption{A cyclic orientation of a \B-cell leading to a double-crossing (left) or the unique\cref{footnote:K33} planar embedding of $K_{3,3}-e$ (right).}
        \label{fig:cyclic-assignment}
    \end{figure}

    Now for every \B-cell $c$, we set $\phi(c)$ to be a crossing~$x$ in $\partial c$ whose two edge-segments in $\partial c$ are oriented outgoing from~$x$.
    Moreover, by \cref{lem:A-cells-vs-X} for every \A-cell $c$, we can set $\phi(c)$ to be a crossing in $\partial c$ such that $\phi(c) \neq \phi(c')$ for any distinct \A-cells $c,c'$.
    
    \begin{claim*}
        The mapping $\phi \colon \calC_4 \to \calX$ is injective.
    \end{claim*}
    \begin{claimproof}
        For a \B-cell $c$ and a \A-cell or \B-cell $c'$ with $\phi(c) = x \in \partial c'$, we shall show $\phi(c') \neq x$.
        
        \begin{figure}[htb]
            \centering
            \includegraphics{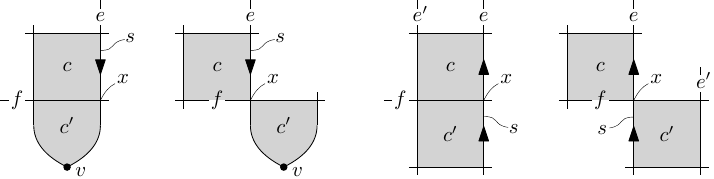}
            \caption{The four cases of a \B-cell $c$ sharing a crossing $x$ with a \A-cell or \B-cell $c'$.}
            \label{fig:injective}
        \end{figure}
        
        If $c'$ is a \A-cell, let $e$ be the edge that is incident to the vertex $v \in \partial c'$ and contains $x$.
        Further, let $f$ be the other edge at $x$ (containing the inner edge-segment of $c'$) and let $s$ be the edge-segment of $e$ in $\partial c$; see \cref{fig:injective}.
        Evidently, $v$ is the common endpoint of all edges crossing $f$.
        In particular, $s$ is oriented inwards at $x$, which is a contradiction to $x = \phi(c)$.
        
        If $c'$ is a \B-cell, let $s$ be an edge-segment that ends at $x$ and belongs to $\partial c'$, but not to $\partial c$.
        Let $e$ be the edge containing $s$, let $f$ be the other edge at $x$, and let $e'$ be the edge containing the edge-segment opposite of $s$ in $\partial c'$; see \cref{fig:injective}.
        As $\phi(c) = x$, the edge-segment of $e$ in $\partial c$ is oriented outwards at $x$ and towards the common endpoint of all edges crossing $f$.
        As $e$ and $e'$ cross $f$, edge-segment $s$ is oriented inwards at $x$ and thus $\phi(c') \neq x$.        
    \end{claimproof}
    
    Clearly, the last claim implies the desired $|\calC_4| \leq |\calX|$.
\end{proof}

Let us prove the edge density of $5n-10$ for connected simple fan-crossing graphs in a 
slightly stronger form.

\begin{theorem}\label{thm:fan-crossing}
    Let $\Gamma$ be a simple connected fan-crossing drawing of some graph $G = (V,E)$ with $|V| \geq 3$.
    Then 
    \[
        |E| \leq 5|V| - 10 - \sum_{c \in \calC_{\geq 5}}(\|c\|-5).
    \]
\end{theorem}
\begin{proof}
    As every edge is crossed only by adjacent edges and adjacent edges do not cross ($\Gamma$ is simple), there are no \C-cells in $\Gamma$ and, hence, $|\calC_3| = 0$.
    Therefore \cref{cor:density-formula-5} (i.e., the Density Formula with $t=5$) immediately gives
    \[
        |E| = 5|V|-10+2|\calC_3|+|\calC_4|-|\calX| - \sum_{c \in \calC_{\geq 5}}(\|c\|-5) \leq 5|V|-10 - \sum_{c \in \calC_{\geq 5}}(\|c\|-5),
    \]
    where the last inequality uses \cref{lem:adj-crossing}.
\end{proof}

\subsection{Flaws in the Original Proofs from Related Work}
\label{sec:fan-planar-gap}

Recall that fan-planar graphs are a special case of fan-crossing graphs, defined by admitting drawings in $\mathbb{R}^2$ without configuration~I and~II (original definition~\cite{KU14-arxiv}), respectively without configurations~I,~II, and~III (revised definition~\cite{KaufUeck22});
cf.\ \cref{fig:fan-planar-mess}.
The proofs in~\cite{KU14-arxiv,KaufUeck22} involve a number of statements, each carefully analysing the possible routing of edges in a fan-planar drawing.
In the past decade, many papers on (generalizations of) fan-planar graphs appeared and many rely (implicitly or explicitly) on said statements.
As mentioned above, a flaw in one of the statements from~\cite{KU14-arxiv} was discovered~\cite{KKRS23}.
In this section, we will describe additional issues existing in both~\cite{KU14-arxiv} and~\cite{KaufUeck22}, thereby outlining why the previous proofs of the density bounds for fan-crossing, weakly fan-planar, and strongly fan-planar graphs are indeed incomplete.

\begin{figure}[htb]
    \centering
    \includegraphics{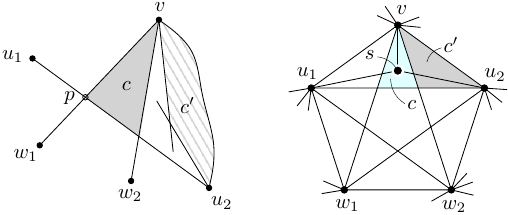}
    \caption{
        Left: Illustration of \cite[Corollary 5]{KaufUeck22} taken from the paper.
        Right: A counterexample.
    }
    \label{fig:fan-planar-problems}
\end{figure}

\begin{itemize}
    \item The authors try to guarantee \cite[Corollary 5]{KaufUeck22} that no cell of size~$4$ of any \emph{subdrawing} of a fan-planar drawing $\Gamma$ contains vertices of $G$.
    In fact, if $c$ is a \A-cell with incident vertex $v$ and inner edge-segment of an edge $u_1u_2$, and some set $S$ of vertices lies inside $c$, it is suggested to move the drawing of $G[S]$ to a cell $c'$ incident to an uncrossed edge $vu_1$ or $vu_2$ as illustrated in \cref{fig:fan-planar-problems}(left).
    However, in the particular situation of \cref{fig:fan-planar-problems}(right) with $S$ just being a single vertex $s$, moving $s$ into $c'$ would cause edge $u_1s$ to cross edge $vw_2$, which is already crossed by the independent edge $u_2w_1$; thus loosing fan-planarity.
    
    \item In a later proof \cite[Lemma 11]{KaufUeck22}, induction is applied to the induced subdrawing of an induced  subgraph~$G'$ of $G$.
    However throughout, the drawing $\Gamma$ was chosen to satisfy (i) having the maximum number of planar edges, and (ii) being inclusionwise edge-maximal with that property~\cite[Section 3]{KaufUeck22}.
    It is not shown or clear why the subdrawing for the induction still satisfies (i) and (ii).
    %
 \end{itemize}

\section{Quasiplanar Graphs}
\label{sec:quasiplanar}

The lower bounds for simple and non-homotopic quasiplanar graphs presented in this section are based on properties that tight examples must have that arise from a thorough reading of our upper bound proof, as provided in \cref{sec:quasiplanar-appendix} and \cref{sec:simple-quasiplanar-appendix}.
For instance, the removal of any vertex leaves a cell of size 2 in the non-homotopic case, while in the simple case, the uncrossed edges must form a matching.

\begin{theorem}\label{thm:LB-general-quasiplanar}
    For every $n \geq 4$, there exists a non-homotopic $n$-vertex connected quasiplanar graph with $8n-20$ edges.
\end{theorem}
\begin{proof}
    For $n = 4$, let us simply refer to the construction illustrated in \cref{fig:LB-general-quasiplanar}(top-left).

    \begin{figure}[htb]
        \centering
        \includegraphics{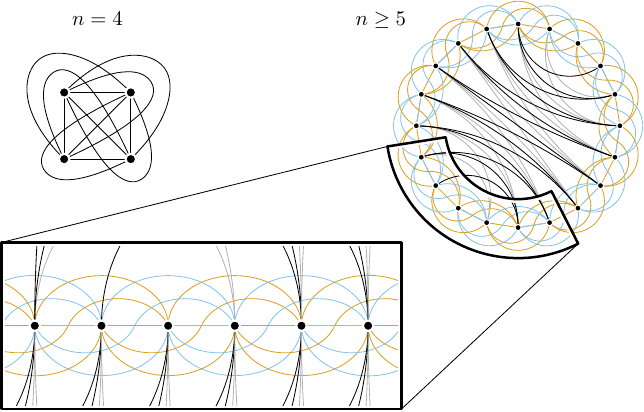}
        \caption{
            Illustrations of non-homotopic quasiplanar drawings with $n$ vertices and $8n-20$ edges.
            For better readability, the two zig-zag paths outside the cycle are omitted.
            The edge-coloring (works only for even $n$) just indicates four crossing-free sub-drawings, which helps to verify quasiplanarity.
        }
        \label{fig:LB-general-quasiplanar}
    \end{figure}

    For $n \geq 5$, the desired graph $G_n$ consists of (for illustrations refer to \cref{fig:LB-general-quasiplanar}(right))
    \begin{itemize}
        \item an $n$-vertex cycle $C$ drawn in a non-crossing way,\hfill ($n$ edges)
        \item an edge between any two vertices at distance~$2$ on $C$ drawn inside $C$,\hfill ($n$ edges)
        \item an edge between any two vertices at distance~$2$ on $C$ drawn outside $C$,\hfill ($n$ edges)
        \item an edge between any two vertices at distance~$3$ on $C$, starting inside $C$,\\
            crossing $C$ at distance~$1.5$, and ending outside $C$,\hfill ($n$ edges)
        \item a zig-zag path of edges drawn inside $C$ where\\
            the endpoints of each edge have distance at least~$3$ on $C$,\hfill ($n-5$ edges)
        \item another (different) zig-zag path of edges drawn inside $C$ where\\
            the endpoints of each edge have distance at least~$3$ on $C$,\hfill ($n-5$ edges)
        \item a zig-zag path of edges drawn outside $C$ where\\
            the endpoints of each edge have distance at least~$3$ on $C$,\hfill ($n-5$ edges)
        \item another (different) zig-zag path of edges drawn outside $C$ where\\
            the endpoints of each edge have distance at least~$3$ on $C$,\hfill ($n-5$ edges)
    \end{itemize}
    Thereby, all edges are drawn without unnecessary crossings.
    For example, two edges drawn inside $C$ cross only if the respective endpoints appear in alternating order around $C$.

    Evidently, $G_n$ has $n$ vertices and $8n-20$ edges, and it is straightforward to check that the described drawing of $G_n$ is non-homotopic and quasiplanar.
\end{proof}


\begin{theorem}\label{thm:LB-simple-quasiplanar}
    For every even $n \geq 8$, there exists a simple $n$-vertex quasiplanar graph with $6.5n-20$ edges.
\end{theorem}
\begin{proof}
    Our construction is a subgraph of the corresponding graph in the proof of \cref{thm:LB-general-quasiplanar};
    see \cref{fig:LB-simple-quasiplanar} for an illustration.
    For every even $n \geq 8$, the desired simple quasiplanar graph $G_n$ 
    is missing all the orange edges depicted in \cref{fig:LB-general-quasiplanar} except the ones at distance 1, i.e.
    \begin{itemize}
        \item the edges between any two black vertices at distance~$2$ on $C$\\
            drawn inside $C$,\hfill ($n/2$ edges)
        \item the edges between any two white vertices at distance~$2$ on $C$\\
            drawn outside $C$,\hfill ($n/2$ edges)
        \item the edges from each black vertex to its white vertex clockwise at distance~$3$\\
            on $C$, starting inside $C$, crossing $C$ at distance~$1.5$, and ending outside $C$,\hfill ($n/2$ edges)
    \end{itemize}
    As $n \geq 8$, the four zig-zag paths can be chosen without introducing parallel edges.
    Again, all edges are drawn without unnecessary crossings.
    For example, two edges drawn inside $C$ cross only if the respective endpoints appear in alternating order around $C$.

    \begin{figure}[htb]
        \centering
        \includegraphics[page=2]{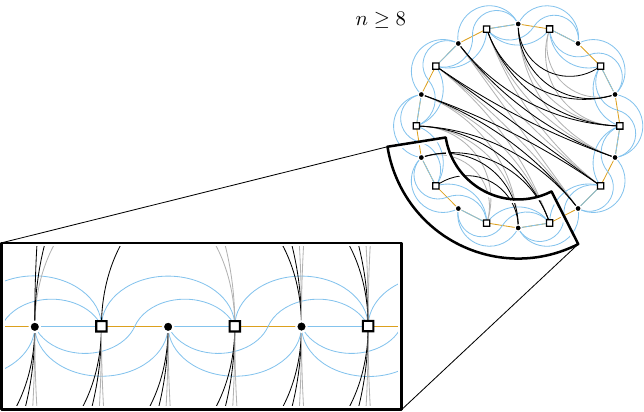}
        \caption{
            Illustration of simple quasiplanar drawings with $n$ vertices and $6.5n-20$ edges, for even $n \geq 8$.
            For better readability, the two zig-zag paths outside the cycle are omitted.
            The edge-coloring just indicates four crossing-free sub-drawings, which helps to verify quasiplanarity.
        }
        \label{fig:LB-simple-quasiplanar}
    \end{figure}

    Evidently, $G_n$ has $n$ vertices and $6.5n-20$ edges and it is straightforward to check that the described drawing of $G_n$ is simple and quasiplanar.    
\end{proof}

\section{Concluding Remarks}
\label{sec:conclusions}

Some previously known proofs 
already contain ideas that are similar to (parts of) our approach.
Often times, this is phrased in terms of a discharging argument, instead of a direct counting.
For example, some discharging steps in~\cite{AckeTard07},~\cite{Ack09},~\cite{Ack19-k-planar}, and~\cite{AckKes23} (dealing with $k$-planar, so-called $k$-quasiplanar graphs, and fan-crossing graphs) directly correspond to our proof of \cref{lem:B-general}.
In these four cases, but also in~\cite{DBLP:conf/wg/BinucciBDHKLMT23}, the total sum of all charges is $\sum_{c \in \calC}(\|c\|-4)$ (although stated a bit differently).
In~\cite{AnBeFoKa20}, which concerns $1$-bend RAC-graphs, there is a charging involving the convex bends.
Further, the concept of the size of a cell and the quantity $\sum_{c \in \calC}(\|c\|-5)$ already appear in the papers~\cite{KU14-arxiv,KaufUeck22} on fan-planar graphs.
But with the Density Formula, we have a unified approach that somewhat unveiled the essential tasks in this field of research.
A valuable asset of our approach are very streamlined and clean combinatorial arguments, as well as substantially shorter proofs, as certified by the number of beyond-planar graph classes that we can treat in about 20 pages.

Additionally, it is straightforward to derive from the particular application of the Density Formula properties that must be fulfilled by all tight examples.
For example, from our proof for $k^+$-real face graphs (in \cref{sec:k-real-face-graphs}), we immediately see that all tight examples of $k^+$-real face graphs with $k\geq 3$ are planar.
Similarly, from our proof for $2$-planar graphs (in \cref{sec:k-planar}), we see that no tight example of a $2$-planar graph has a \C-cell or \B-cell.
And (together with a short calculation) our proof for quasiplanar graphs (in \cref{sec:simple-quasiplanar-appendix}) implies that in all tight examples of simple quasiplanar graphs the planar edges form a perfect matching.
Specifically, this approach of analysing the situation in which the proof with the Density Formula is tight, allowed us to find the first tight examples for simple and non-homotopic quasiplanar graphs (cf.~\cref{thm:LB-general-quasiplanar,thm:LB-simple-quasiplanar}).


The only cases presented here in which upper and lower bounds still differ by an absolute constant are $k$-bend RAC-graphs; cf.~\cref{tab:overview}.
This is due to the fact that these are drawings in the plane $\mathbb{R}^2$ and the unbounded cell behaves crucially different from all other cells.
It is possible to reduce our upper bounds by an absolute constant by a separate analysis of the unbounded cell, but we did not pursue this here.
On the other hand, it may well be that our bounds are already optimal for $k$-bends RAC-drawings on the sphere~$\mathbb{S}^2$ (for the natural definition of this concept)
 ---they definitely are for $k=0$.

Finally for open problems, there is a number of beyond-planar graph classes for which the exact asymptotics of their edge density is not known yet.
This includes for example non-homotopic fan-crossing graphs, $k$-quasiplanar graphs for $k \geq 4$, and $k$-planar graphs for $k \geq 4$.
Let us refer again to the survey~\cite{DLM19-survey} from 2019 for more such cases and more beyond-planar graph classes in general.
Additionally, each class could be considered in a ``bipartite variant'' (as we do for fan-crossing graphs in \cref{thm:bip-fan-crossing}) and/or an ``outer variant'' where one additionally requires that there is one cell that is incident to every vertex; see for example~\cite{ABKPU18}. 


\newpage

\bibliographystyle{plainurl}
\bibliography{lit}

\newpage
\section{Appendix}
\label{sec:appendix}

\subsection{\texorpdfstring{$\boldsymbol{k}$}{k}-Bend RAC-Graphs}
\label{sec:k-bend-RAC-graphs-appendix}

For completeness, let us also apply the Density Formula to $0$-bend RAC-graphs.
As edges are straight segments, every $0$-bend RAC-drawing is simple and contains no degenerate cells.

\begin{theorem}\label{thm:0-bend-RAC}
    For every $n \geq 3$, every $n$-vertex connected $0$-bend RAC-graph $G$ has at most $4n-8$ edges.
\end{theorem}
\begin{proof}
    Let $\Gamma$ be a $0$-bend RAC-drawing of $G = (V,E)$.
    As every crossing has a right angle, we have no \A-cells, no \C-cells, and no \F-cells.
    Hence $\calC_3 = \emptyset$, $\calC_4$ is the set of all \B-cells, and $\calC_5$ is the set of all \D-cells and all \E-cells.
    Now, by \cref{obs:inner-segments} we have $2|\calX| - |E_x| = |\calS_{\rm in}|$ and by \cref{lem:B-general} we have $|\calS_{\rm in}| \geq 2 \cdot \#\text{\B-cells}$.
    Hence
    \begin{equation}
        \#\text{\B-cells} - |\calX| \leq -\frac12 |E_x|.\label{eq:RAC-ineq1}
    \end{equation}
    Moreover, every \D-cell and every \E-cell has exactly two incident outer edge-segments of some crossed edge.
    As every crossed edge has exactly two outer edge-segments and every edge-segment has exactly two incident cells, it follows that
    \begin{equation}
        2|\calC_5| = 2\#\text{\D-cells} + 2\#\text{\E-cells} \leq 4|E_x|.\label{eq:RAC-ineq2}
    \end{equation}
    Thus, the Density Formula with $t = 4$ (\cref{cor:density-formula-4}) immediately gives
    \begin{align*}
        |E| &\leq 4|V|-8 + \frac74|\calC_3|+|\calC_4|+\frac14|\calC_5|-|\calX| = 4|V|-8 + \#\text{\B-cells}+\frac14|\calC_5|-|\calX|\\
            &\overset{\eqref{eq:RAC-ineq2}}{\leq} 4|V|-8 + \#\text{\B-cells} + \frac12|E_x| - |\calX| \overset{\eqref{eq:RAC-ineq1}}{\leq} 4|V|-8.\qedhere
    \end{align*}
\end{proof}

\subsection{Bipartite Fan-Crossing Graphs}
\label{sec:bipartite-fan-crossing}

Using lemmas from~\cite{KU14-arxiv}, the authors of~\cite{ABKPU18} proved in 2018 that simple $n$-vertex bipartite strongly fan-planar graphs have at most $4n-12$ edges while providing a lower bound example with $4n-16$ edges.
The same upper bound of $4n-12$ is claimed for weakly fan-planar graphs, again by reducing to the strongly fan-planar case~\cite{CFKPS23}.
For non-homotopic bipartite fan-planar graphs, there is a $4n-12$ lower bound~\cite{ABKPU18} but no upper bound. 


In this section, we prove that simple connected $n$-vertex bipartite fan-crossing graphs have at most $4n-10$ edges. 
As the upper bounds in~\cite{ABKPU18} and~\cite{CFKPS23} rely on incorrect statements in~\cite{KU14-arxiv} and~\cite{KaufUeck22}, our result also provides the first complete proof for the special case of bipartite weakly and strongly fan-planar graphs.

\begin{lemma}\label{lem:bipartite-fan-crossing}
    Let $\Gamma$ be a simple connected fan-crossing drawing of a bipartite graph $G = (V,E)$ with $|V| \geq 3$.
    Then $|V| \leq \sum_{c \in \calC_{\geq 5}} (\|c\|-5)$.
\end{lemma}
\begin{proof}
    First, as $\Gamma$ is simple and $|V| \geq 3$, there are no degenerate cells with at most four edge-segment-incidences.
    We shall map each vertex $v \in V$ onto one of its incident cells $c = \phi(v)$ in such a way that no cell $c$ is used more than $\|c\|-5$ times.
    In particular, for each vertex $v$ the cell $\phi(v)$ must be of size at least~$6$.
    
    First, for every cell $c$ with at most $\|c\|-5$ vertex-incidences (equivalently, at least five edge-segments-incidences), we set $\phi(v) = c$ for each vertex $v \in \partial c$.
    If a vertex is incident to more than one such cell, we pick one arbitrarily.
    To map the remaining vertices, we now analyze the types of cells that can(not) occur in~$\Gamma$.

    Let the bipartition of $G$ be $V = W \dot\cup B$ with an independent set $W$ of \emph{white} vertices and an independent set $B$ of \emph{black} vertices.
    For each cell $c$ let us encode the occurrences of white vertices ($W$), black vertices ($B$), and edge-segments ($\calS$) around $\partial c$ as a cyclic sequence $\partial c = [x_1,\ldots,x_k]$ with $k = \|c\|$ and each $x_i \in \{W,B,\calS\}$.

    \begin{claim*}
        There is no cell $c$ in $\Gamma$ with $\sigma_c = [\calS,\calS,\calS]$ or $\sigma_c = [W,\calS,B,\calS,\calS,\calS]$.
    \end{claim*}
    \begin{claimproof}
        The sequence $\sigma_c = [\calS,\calS,\calS]$ would correspond to 
        a \C-cell, which is impossible in a simple fan-crossing drawing.

        For $\sigma_c = [W,\calS,B,\calS,\calS,\calS]$, let $w \in W$ and $b \in B$ be the vertices in $\partial c$ and let $e$ be the edge in $\partial c$ incident to neither $w$ nor $b$; see \cref{fig:bip-fan-crossing-impossible}(second).
        Then $e$ is crossed by an edge $f_1$ with endpoint $w$ and an edge $f_2$ with endpoint $b$.
        As $wb$ is already an edge in $\partial c$ that is distinct from $f_1$ and $f_2$, the common endpoint of $f_1,f_2$ must be a third vertex $v$, which makes $wbv$ a 3-cycle in $G$ -- a contradiction to the fact that $G$ is bipartite.
    \end{claimproof}

    \begin{figure}[htb]
        \centering
        \includegraphics{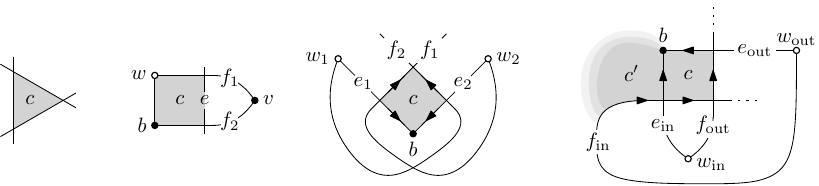}
        \caption{Illustrations of impossible types of cells $c$ in $\Gamma$ (the right-most one can occur under certain assumptions about $c'$; the others do not occur at all).}
        \label{fig:bip-fan-crossing-impossible}
    \end{figure}

    To exclude another type of cell, we now orient every edge-segment in $\Gamma$ towards the black endpoint of its edge.

    \begin{claim*}
        There is no cell $c$ in $\Gamma$ with $\sigma_c = [B,\calS,\calS,\calS,\calS]$ and with two edge-segments in $\partial c$ oriented towards a common crossings in $\partial c$.
    \end{claim*}
    \begin{claimproof}
        Assume for the sake of contradiction that $c$ is such a cell whose boundary consists (in that cyclic order) of $b \in B$, and edge-segments of edges $e_1 = bw_1$, $f_1$, $f_2$, and $e_2 = bw_2$; see \cref{fig:bip-fan-crossing-impossible}(third).
        If the common endpoint of $f_1$ and $e_2$ (they both cross $f_2$) is $b$, then $f_1$ and $e_1$ are adjacent crossing edges -- a contradiction to the simplicity of $\Gamma$.
        So the common endpoint of $f_1$ and $e_2$ is $w_2$ and, symmetrically, the common endpoint of $f_2$ and $e_1$ is $w_1$.
        But now, the assumed orientation of edge-segments forces $f_1$ and $f_2$ to cross a second time.
    \end{claimproof}

    \cref{fig:bip-fan-crossing-cells} shows all remaining possible cases of a cell $c$ with an incident black vertex $b \in B$ and at most four edge-segments-incidences.
    If $c$ is such a cell for which $\sigma_c = [B,\calS,\calS,\ldots]$, i.e., there is an edge-segment in $\partial c$ between $b$ and a crossing $x$, let $s$ be the other edge-segment at $x$ in $\partial c$ and call $x$ \emph{incoming at $c$} if $s$ is oriented outwards at $x$ and call $x$ \emph{outgoing at $c$} if $s$ is oriented inwards at $x$.
    Let us call a cell of size $5$ with one incoming and one outgoing crossing a \emph{bad cell} and a cell of size $6$ with two outgoing crossings a \emph{good cell}. 

    \begin{figure}[htb]
        \centering
        \includegraphics{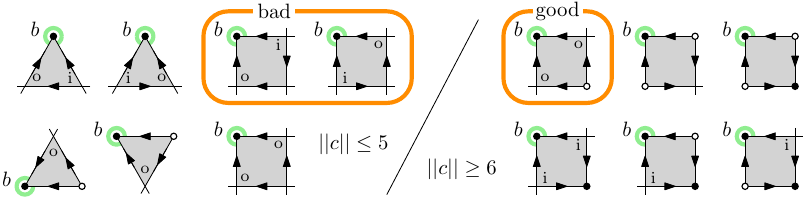}
        \caption{All possible situations of a cell with an incident black vertex $b$ and at most four edge-segments.
        The orientation of edge-segments towards black endpoints is indicated by arrows.
        Incoming crossings are labeled `i', outgoing crossings are labeled `o'.
        }
        \label{fig:bip-fan-crossing-cells}
    \end{figure}

    \begin{claim*}
        If a black vertex $b$ is incident to a bad cell $c$, then $b$ is also incident to a good cell or a cell with at least five edge-segment-incidences.
    \end{claim*}
    \begin{claimproof}
        Let $e_{\rm in},f_{\rm in}$ be the two edges crossing at the incoming crossing at $c$ with $e_{\rm in} = bw_{\rm in}$ being incident to $b$.
        Similarly, let $e_{\rm out},f_{\rm out}$ be the two edges crossing at the outgoing crossing at $c$ with $e_{\rm out} = bw_{\rm out}$ being incident to $b$; see \cref{fig:bip-fan-crossing-impossible}(fourth).
        If the common endpoint of $e_{\rm in}$ and $f_{\rm out}$ (they both cross $f_{\rm in}$) is $b$, then $f_{\rm out}$ and $e_{\rm out}$ are crossing adjacent edges -- a contradiction.
        Thus, $f_{\rm out}$ is incident to $w_{\rm in}$.
        Symmetrically, $f_{\rm in}$ is incident to $w_{\rm out}$.

        Now consider the cell $c'$ incident to $b$ sharing an edge-segment of $e_{\rm in}$ with $c$.
        If $c'$ has at most four edge-segments (otherwise we are done), then $c'$ appears in \cref{fig:bip-fan-crossing-cells}.
        We already know that $\sigma_{c'} = [B,\calS,\calS,\ldots]$ with both edge-segments oriented consistently.
        If $\sigma_{c'} = [B,\calS,\calS,W,\calS]$,
        then there would be an uncrossed edge between $b$ and $w_{\mathrm{out}}$, contradicting the fact that there also is a crossed version of this edge.
        If $\sigma_{c'} = [B,\calS,\calS,\calS]$, i.e., $f_{\rm in}$ is crossed by another edge $e$ incident to $b$, then $e$ must also be incident to $w_{\rm in}$ and, hence, be parallel to $e_{\rm in}$ -- a contradiction.
        If $\sigma_{c'} = [B,\calS,\calS,\calS,\calS]$, then $f_{\rm in}$ is crossed by another edge $e_1$, which is crossed by yet another edge $e_2$ that contains an outer edge-segment incident to $b$.
        Note that this outer edge-segment cannot be incident to $c$ (i.e., $e_2\neq e_{\rm out}$) since this would imply that $e_1,f_{\rm in},f_{\rm out}$ pairwise cross, which cannot happen in a simple fan-crossing drawing.
        If the common endpoint of $e_2$ and $f_{\rm in}$ (they both cross $e_1$) is $b$, then $f_{\rm in}$ and $e_{\rm in}$ are crossing adjacent edges -- a contradiction.
        Thus, $e_2$ is incident to $w_{\rm out}$ and, hence, parallel to $e_{\rm out}$ -- a contradiction.
        It follows that $\sigma_{c'} = [B,\calS,\calS,W,\calS,\calS]$, i.e., $c'$ is a good cell.
    \end{claimproof}

    \begin{claim*}
        Every black vertex $b$ is incident to a cell $c$ with at least five edge-segment-incidences, or to at least two cells $c_1,c_2$ of size at least $6$. 
    \end{claim*}
    \begin{claimproof}
        Assume that every cell incident to $b$ has at most four edge-segment-incidences, i.e., is depicted in \cref{fig:bip-fan-crossing-cells}.
        Observe that each $c$ with $\|c\| \leq 5$ has at least as many outgoing as incoming crossings.
        Moreover, the only cells $c$ with $\|c\| \geq 6$ with more incoming than outgoing crossings are the three in the bottom row.
        In total, there is the same number of incoming and outgoing crossings around $b$.
        Thus, if $b$ is incident to a good cell, then $b$ is also incident to a second cell of size at least~$6$, as desired.

        \begin{figure}[htb]
            \centering
            \includegraphics{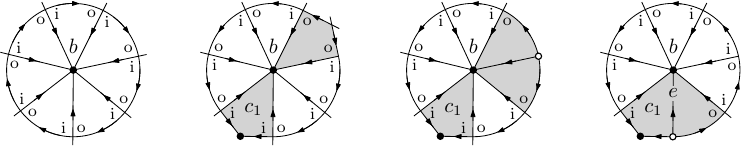}
            \caption{
                All combinations of non-good cells with at most four edge-segments around vertex $b$ with at most one such cell $c_1$ of size $\|c_1\| \geq 6$ give parallel edges, crossing adjacent edges or an edge without endpoints.
                Cells with a different number of incoming and outgoing crossings are highlighted.
            }
            \label{fig:bip-fan-crossing-neighborhood}
        \end{figure}

        So assume there is no good cell and, hence, by the previous Claim also no bad cell at $b$.
        Having only cells with exactly one incoming and one outgoing crossing (but no bad cell) would give an edge without endpoints as illustrated in \cref{fig:bip-fan-crossing-neighborhood}(first) -- a contradiction.
        It follows that there is at least one cell $c_1$ of size at least~$6$ at $b$ (and it is not good).
        Assume for the sake of contradiction that all cells at $b$ other than $c_1$ have size at most~$5$.
        First, say $c_1$ has two incoming crossings.
        If this is compensated by a cell with two outgoing crossings, we have crossing adjacent edges and if by two cells with no incoming crossing, we have two parallel edges, so we have a contradiction in both cases;
        see \cref{fig:bip-fan-crossing-neighborhood}(second, third).

        It remains to consider the case that $c_1$ has an uncrossed edge $e$ incident to $b$.
        Such $e$ must be shared by a cell with exactly one outgoing crossing, forcing $c_1$ to have an incoming crossing.
        But then we have two parallel edges; see \cref{fig:bip-fan-crossing-neighborhood}(fourth).
    \end{claimproof}

    Finally, the last Claim allows us to complete the desired mapping $\phi \colon V \to \calC$ with the property that each cell $c$ is used at most $\|c\|-5$ times.
    For this, it remains to set $\phi(v)$ for vertices only incident to cells with at most four edge-segment-incidences (displayed in right of \cref{fig:bip-fan-crossing-cells}).
    By the last Claim every black (and by symmetry also every white) such vertex is incident to at least two such cells with size at least $6$.
    
    As every such cell $c$ has $\|c\|-4$ vertex-incidences, we can set all $\phi(v)$ so that each cell $c$ is used at most $\|c\|-5$ times:
    This is done by going through the unassigned vertices and setting $\phi(v)$ for the current $v$ to one of its incident large cells.
    Whenever a large cell $c$ is used the maximum number of $\|c\|-5$ times, we continue with the only vertex $v$ incident to $c$ such that $\phi(v)\neq c$.
    If $\phi(v)$ is not set already, we set $\phi(v)$ to the other incident cell at $v$ of size at least $6$, and continue as before.
    Since at all times there is at most one cell $c$ that is used its maximum number of $\|c\|-5$ times while being incident to an unassigned vertex, setting $\phi(v)$ for the next vertex $v$ is always valid.
    Eventually, this gives the desired mapping.
\end{proof}

\begin{theorem}\label{thm:bip-fan-crossing}
    For every $n \geq 3$, every simple connected $n$-vertex bipartite fan-crossing graph $G$ has at most $4n-10$ edges.
\end{theorem}
\begin{proof}
    Let $\Gamma$ be a simple fan-crossing drawing of a bipartite graph $G = (V,E)$.
    Then \cref{thm:fan-crossing} and \cref{lem:bipartite-fan-crossing} immediately give
    \[
        |E| \leq 5|V|-10 - \sum_{c \in \calC_{\geq 5}} (\|c\|-5) \leq 4|V|-10.\qedhere
    \]
\end{proof}

\subsection{Quasiplanar Graphs}
\label{sec:quasiplanar-appendix}

In this section, we reprove the upper bounds for connected simple and non-homotopic quasiplanar graphs by applications of the Density Formula.


As it will be needed, let us start with a general lemma that holds for every non-homotopic drawing (not necessarily quasiplanar).
Loosely speaking, for this lemma we want to remove a vertex $v$ from a drawing $\Gamma$, consider the ``cell it leaves behind'', and in particular compute the size of that cell.
Formally, the \emph{link of $v$} is the new cell resulting from subdividing all crossed edges incident to $v$ such that the part of the edge between the subdivision vertex and $v$ has exactly one crossing, then removing $v$ and all its incident edges from the obtained drawing.

\begin{lemma}\label{lem:remove-vertex}
    Let $\Gamma$ be any non-homotopic drawing of some connected graph $G$ on at least three vertices and $v$ be a vertex of $G$.
    Let $\calC(v) \subseteq \calC$ be the set of all cells incident to $v$, and $c_0$ be the link of $v$.
    Then $\|c_0\| = \sum_{c \in \calC(v)}(\|c\|-5) + |\calC(v)|$.
\end{lemma}
\begin{proof}
    Let $\mathcal{I}_v$ be the multiset that contains a tuple $(a,c)$ for each cell $c\in \mathcal C(v)$ and each incidence of a vertex or edge-segment $a$ with $c$. So, if $a$ appears multiple times on $\partial c$, then $\mathcal{I}_v$ contains the tuple $(a,c)$ with the according multiplicity.
    Let $A \subseteq \mathcal{I}_v$ be the tuples $(a,c)$ where $a$ corresponds to the vertex $v$ or an edge-segment incident to $v$.
    Then $|A| = 3 \deg(v)$.

    Now, since $\Gamma$ is non-homotopic and connected, the link $c_0$ of $v$ is the union of all cells in $\calC(v)$.
    Consider any vertex-incidence $i$ of a vertex $w$ with the cell $c_0$.
    In $\Gamma$, there is some number $b \geq 0$ of (uncrossed) edge-segments that are incident to $v$ and end at the incidence $i$.
    These correspond to exactly $b+1$ tuples $(w,c) \in \mathcal{I}_v$.
    Similarly, consider any edge-segment-incidence $i$ of an edge-segment and the cell $c_0$.
    In $\Gamma$, there is some number $b \geq 0$ of (uncrossed) edge-segments that are incident to $v$ and end at the incidence $i$.
    Then $i$ corresponds to exactly $b+1$ edge-segment-incidences in $\Gamma$ and, hence, there are exactly $b+1$ according tuples in $\mathcal{I}_v$.
    So, loosely speaking, for each of the $\deg(v)$ incident edge-segments at $v$, we have one extra tuple in $\mathcal{I}_v$ compared to the vertex-incidences and edge-segments-incidences of $c_0$.
    Thus, in total, the claim follows by:
    \[
        \sum_{c \in \calC(v)} \|c\| = |\mathcal{I}_v| = |A| + \|c_0\| + \deg(v) = \|c_0\| + 4\deg(v) = \|c_0\| + 4|\calC(v)|.\qedhere
    \]
\end{proof}

Turning back to quasiplanar drawings, we begin with an easy observation.

\begin{observation}\label{obs:quasiplanar-no-3-cells}
    A quasiplanar drawing has no \C-cells, and thus $|\calC_3| = 0$.\qed
\end{observation}
%

A drawing $\Gamma$ is called \emph{filled}~\cite{CKPRU21} if, whenever two vertices $u,v$ are incident to the same cell $c$ in $\Gamma$, then $uv$ is an uncrossed edge in the boundary of $c$.

\begin{lemma}\label{lem:general-quasiplanar-filled}
    Every non-homotopic quasiplanar graph $G$ can be augmented by only adding edges to a graph that admits a filled non-homotopic quasiplanar drawing.
\end{lemma}
\begin{proof}
    Let $\Gamma$ be a non-homotopic quasiplanar drawing of $G$.
    Assume that $\Gamma$ is not filled, and that $u,v$ are two vertices incident to the same cell $c$ in $\Gamma$, but not connected with an edge along $\partial c$.
    We draw an uncrossed edge from $u$ to $v$ through $c$.
    The new drawing is non-homotopic:
    this is clearly the case if there was not already an edge between $u$ and $v$ in $G$.
    Otherwise, the cyclic sequence $\partial c$ that contains $u$ and $v$ has a crossing or vertex in both open intervals delimited by $u$ and $v$.
    Since this procedure increases the number of edges, it can be repeated to end up with a filled drawing.
\end{proof}

\begin{observation}\label{obs:kMinusReal}
    In a filled drawing~$\Gamma$, no cell is incident to more than three pairwise distinct vertices.
    Further, no cell in~$\Gamma$ is incident to more than two pairwise distinct vertices if (and only if) there are no \T-cells in $\Gamma$.\qed
\end{observation}

\begin{lemma}\label{lem:no-T-cells}
    Every non-homotopic connected quasiplanar graph $G$ on at least four vertices can be augmented by only adding edges to a graph that admits a filled non-homotopic quasiplanar drawing with no \T-cells.
\end{lemma}
\begin{proof}
    By \cref{lem:general-quasiplanar-filled}, we may assume that $G$ has a filled non-homotopic quasiplanar drawing $\Gamma$.
    Assume that $c$ is a \T-cell in $\Gamma$ and in clockwise order $v_0,v_1,v_2$ are the three vertices around $\partial c$.

    \begin{figure}[htb]
        \centering
        \includegraphics{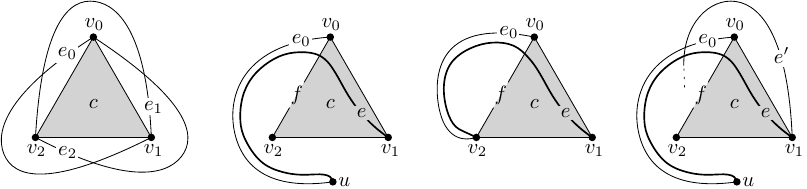}
        \caption{Illustrations of the proof of \cref{lem:no-T-cells}.}
        \label{fig:no-T-cell}
    \end{figure}

    For $i\in\{0,1,2\}$ (all indices modulo~$3$), if $\deg(v_i) > 2$, then let $e_i$ be the edge incident to $v_i$ that clockwise succeeds $c$ in the cyclic order of edges incident to~$v_i$.
    See \cref{fig:no-T-cell}.
    As $G$ is connected and has at least four vertices, at least one such $e_i$ exists. Note that if $e_i$ ends at $v_{i+1}$, then $\deg(v_{i+1}) > 2$ and 
    $e_{i+1}$ exists.
    As $\Gamma$ is quasiplanar, at least one of the $e_i$ ends in some vertex~$u$ other than $v_i$ and $v_{i+1}$, say $e_0 = v_0u$ with $u\neq v_0,v_1$, see \cref{fig:no-T-cell}(first).
    
    Now we draw a new edge $e$ starting at $v_1$, going through $c$, crossing $\partial c$ at the edge $f = v_0v_2$, and following close to $e_0$ to vertex $u$, see \cref{fig:no-T-cell}(second, third).
    Note that $e$ crosses only $f$ and edges that also cross $e_0$ and, hence, the obtained drawing $\Gamma'$ is again quasiplanar. 
   
    To argue that $\Gamma'$ is non-homotopic, let us first consider the case that there is an edge $e' = v_1u$ in $\Gamma'$ parallel to $e$ but not crossing $e$ and consider the two lenses between $e$ and $e'$. 
    If $u \neq v_2$, since $f$ is not crossed by $e'$, one lens contains $v_0$ and the other $v_2$. 
    If $u=v_2$, then $e_0$ and $f$ are parallel edges and, thus, some crossings or vertices are inside both their lenses, see \cref{fig:no-T-cell}(third). 
    By construction, one of the sides of the lens of $e$ and $e'$ contains all the crossings and vertices in one of the two lenses of $e_0$ and $f$, while the other contains $v_0$.

    It remains to consider the case that there is some edge $e'$ that crosses $e$ such that it forms a lens $\ell$ with $e$.
    First, assume that~$\partial \ell$ is incident to the endpoint~$v_1$ of~$e$ (and exactly one crossing of $e'$ with $e$).
    In this case, the lens~$\ell$ contains either~$v_0$ or~$v_2$ since~$f$ is planar in $\Gamma$, see \cref{fig:no-T-cell}(fourth).
    Second, assume that~$\partial \ell$ is either incident to $u$ (and exactly one crossing of $e'$ with $e$) or incident to exactly two crossings of $e'$ with $e$.
    Either way, lens~$\ell$ is essentially a duplicate of a lens between~$e_0$ and~$e'$ and, hence, it contains a crossing or vertex.
    Altogether, this shows that~$\Gamma'$ is indeed non-homotopic.
    
    Finally observe that $\Gamma'$ is again filled.
    By repeating this procedure, we eventually eliminate all \T-cells.
\end{proof}


\begin{lemma}\label{lem:non-homotopic-quasiplanar}
    Let $\Gamma$ be a non-homotopic filled drawing of a connected graph $G = (V,E)$ with $|V|\ge 3$ and where no cell in~$\Gamma$ is incident to more than $2$ pairwise distinct vertices.    
    Then
    \[
        \#\text{\A-cells} \leq 2|E_x|  - 2|V| + 2 \cdot \smashoperator{\sum_{c \in \calC_{\geq 5}}}(\|c\|-5).
    \]
\end{lemma}
\begin{proof}
    For any fixed vertex $v$, let $\calC(v)$ be the set of all cells in $\Gamma$ incident to $v$, and $c_0$ be the link of $v$.
    As $\Gamma$ is filled, each vertex in $\partial c_0$ is a neighbor of $v$ in the planar subgraph $G_p = (V,E_p)$ of $G$.
    Clearly, $\partial c_0$ contains at least two edge-segment-incidences and, hence, $\|c_0\| \geq 2 + \deg_{G_p}(v)$. 
    Thus, by \cref{lem:remove-vertex}
    \begin{equation}
        |\calC(v)| + \sum_{c \in \calC(v)}(\|c\|-5) = \|c_0\| \geq 2 + \deg_{G_p}(v).\label{eq:c_0-at-least-2}
    \end{equation}
    Using that every cell $c \in \calC$ is incident to at most two distinct vertices and summing~\eqref{eq:c_0-at-least-2} over all vertices, we obtain
    \begin{align*}
        \sum_{v \in V}|\calC(v)| - \#\text{\A-cells} + 2\cdot \smashoperator{\sum_{c \in \calC_{\geq 5}}}(\|c\|-5) &\geq 2|V| + 2|E_p|.
    \end{align*}
    Together with $\sum_{v \in V}|\calC(v)| \le \sum_{v\in V}\deg(v) = 2|E| = 2|E_p| + 2|E_x|$ this gives the desired:
    \[
        \#\text{\A-cells} \leq 2|E_x| + 2\cdot \smashoperator{\sum_{c \in \calC_{\geq 5}}}(\|c\|-5) - 2|V|.\qedhere
    \]    
\end{proof}

\begin{theorem}\label{thm:general-quasiplanar}
    For every $n \geq 3$, every connected non-homotopic $n$-vertex quasiplanar graph $G$ has at most $8n-20$ edges.
\end{theorem}
\begin{proof}
    Let $\Gamma$ be a non-homotopic quasiplanar drawing of $G = (V,E)$.
    As $G$ is connected, so is $\Gamma$.
    By \cref{obs:kMinusReal,lem:general-quasiplanar-filled,lem:no-T-cells} we may assume that $\Gamma$ is filled and no cell in~$\Gamma$ is incident to more than two pairwise distinct vertices and, hence, \cref{lem:non-homotopic-quasiplanar} applies.
    Recall from \cref{obs:quasiplanar-no-3-cells} that $|\calC_3|=0$.      
    Now the Density Formula with $t=5$ (\cref{cor:density-formula-5}) gives
    \begin{align*}
        |E|&=5|V|-10 + |\calC_4| - |\calX| -\sum_{c \in \calC_{\geq 5}} (\|c\|-5)\\
        &\leq 5|V|-10+\frac12\left(|\calS_{\rm in}|+\#\text{\A-cells}\right)-|\calX|- \smashoperator{\sum_{c \in \calC_{\geq 5}}} (\|c\|-5) \hspace{4.7em} \text{(by \cref{lem:B-general})}\\
        &\leq 5|V|-10+\frac12\left(\#\text{\A-cells}-|E_x|\right) - \smashoperator{\sum_{c \in \calC_{\geq 5}}} (\|c\|-5) \hspace{5.4em} \text{(by \cref{obs:inner-segments})}\\
        &\leq 5|V|-10+\frac12\left(|E|-2|V| +2\cdot \smashoperator{\sum_{c \in \calC_{\geq 5}}} \|c\|-5\right) - \smashoperator{\sum_{c \in \calC_{\geq 5}}} (\|c\|-5)\hspace{2.2em} \text{(by \cref{lem:non-homotopic-quasiplanar})}\\
        &= 0.5|E| + 4|V| - 10,
    \end{align*}
    which implies the desired $|E| \leq 8|V|-20$.
\end{proof}

\subsection{Simple Quasiplanar Graphs}
\label{sec:simple-quasiplanar-appendix}

The case of simple quasiplanar graphs mostly follows the procedure for non-homotopic ones.
Among the few (necessary) differences, is the $3$-connectivity requirement in the following.

\begin{lemma}\label{lem:simple-quasiplanar-filled}
    Every $3$-connected simple quasiplanar graph $G$ can be augmented by only adding edges to a graph that admits a filled simple quasiplanar drawing $\Gamma$.
\end{lemma}
\begin{proof}
    As in the proof of \cref{lem:general-quasiplanar-filled}, whenever two vertices $u,v$ are incident to $c$ but not connected by an edge along $\partial c$, we draw an uncrossed edge $uv$ through $c$.
    If there already was an edge $e$ between $u$ and $v$, we remove $e$.
    Note that in this case~$e$ has to be crossed, as otherwise $G-\{u,v\}$ would be disconnected, contradicting the $3$-connectivity.
    Hence, this procedure decreases the number of crossings and it can be repeated to end up with the desired drawing.
\end{proof}

The analogous statement to \cref{lem:non-homotopic-quasiplanar} is weaker and requires minimum degree at least~$4$.

\begin{lemma}\label{lem:simple-quasiplanar}
    Let $\Gamma$ be a filled simple quasiplanar drawing of a connected graph $G = (V,E)$ of minimum degree at least~$4$.
    Then
    \[
        \#\text{\A-cells} - |E_x| \leq |E| + 2 \cdot \smashoperator{\sum_{c \in \calC_{\geq 5}}}(\|c\|-5) - 3.5|V|.
    \]
\end{lemma}
\begin{proof}
    For any fixed vertex $v$, let $\calC(v)$ be the set of all cells in $\Gamma$ incident to $v$, and $c_0$ be the link of $v$.
    By \cref{lem:remove-vertex} we have
    \begin{equation}
        |\calC(v)| + \sum_{v \in \calC(v)}(\|c\|-5) = \|c_0\|.\label{eq:c_0-size}
    \end{equation}
    As $\Gamma$ is filled, each vertex in $\partial c_0$ is a neighbor of $v$ in the planar subgraph $G_p = (V,E_p)$ of $G$.

    \begin{claim*}
        We have $\|c_0\| \geq 3 + \deg_{G_p}(v)$, as well as $\|c_0\| \geq 4 + \#\{\text{\T-cells at $v$}\}$.
    \end{claim*}
    \begin{claimproof}
        As $\Gamma$ is simple, $\partial c_0$ has at least three edge-segment-incidences and, hence, $\|c_0\| \geq 3 + \deg_{G_p}(v)$.

        If $\deg_{G_p}(v) > \#\{\text{\T-cells at $v$}\}$, then also the second inequality follows.
        
        If $\deg_{G_p}(v) = 0$, then $c_0$ has no incident vertices, and in particular $\#\{\text{\T-cells at $v$}\} = 0$.
        As $\Gamma$ is quasiplanar, $c_0$ is no \C-cell and hence $\|c_0\| \geq 4 = 4 + \#\{\text{\T-cells at $v$}\}$, as desired.
        
        It remains to consider the case that $\deg_{G_p}(v) = \#\{\text{\T-cells at $v$}\} > 0$, which implies that $v$ is incident only to \T-cells.
        But in this case, using $\deg(v) \geq 4$, we can also conclude that $\|c_0\| = \deg(v) + \#\{\text{\T-cells at $v$}\} \geq 4 + \#\{\text{\T-cells at $v$}\}$.
    \end{claimproof}

    Now recall that every non-\T-cell is incident to at most two vertices, since $\Gamma$ is filled.
    From \cref{obs:quasiplanar-no-3-cells} we have $|\calC_3|=0$, i.e., $\calC_{\geq 5} = \calC - \calC_4$.
    Moreover, $\calC_4$ is the set of all \A-cells and \B-cells, where the latter are not incident to any vertex in $G$.
    Finally, let us denote by $\calT$ the set of all \T-cells in $\Gamma$.
    Then, putting the lower bounds on $\|c_0\|$ from the above Claim into \eqref{eq:c_0-size} and summing over all vertices $v$, we obtain
    \begin{align*}
        \sum_{v \in V} |\calC(v)| - \#\text{\A-cells} + 2 \cdot \smashoperator{\sum_{c \in \calC - (\calC_4 \cup \calT)}}(\|c\|-5) + 3\cdot \#\text{\T-cells} &\geq 3|V| + 2|E_p| \quad \text{and}\\
        \sum_{v \in V} |\calC(v)| - \#\text{\A-cells} + 2 \cdot \smashoperator{\sum_{c \in \calC - (\calC_4 \cup \calT)}}(\|c\|-5) + 3\cdot \#\text{\T-cells} &\geq 4|V| + 3\cdot \#\text{\T-cells}.
    \end{align*}
    Together with $\sum_{v \in V}|\calC(v)| \le \sum_{v\in V}\deg(v) = 2|E| = 2|E_p| + 2|E_x|$ this gives
    \begin{align}
        \#\text{\A-cells} &\leq 2|E_x| + 2 \cdot \smashoperator{\sum_{c \in \calC - (\calC_4 \cup \calT)}}(\|c\|-5) + 3\cdot \#\text{\T-cells} - 3|V| \quad \text{and}\label{eq:quasiplanar-A-cell-count-1}\\
        \#\text{\A-cells} &\leq 2|E| + 2 \cdot \smashoperator{\sum_{c \in \calC - (\calC_4 \cup \calT)}}(\|c\|-5) - 4|V|.\label{eq:quasiplanar-A-cell-count-2}
    \end{align}
    Adding \eqref{eq:quasiplanar-A-cell-count-1} and \eqref{eq:quasiplanar-A-cell-count-2} gives the desired:
    \begin{align*}
        2\cdot \#\text{\A-cells} &\leq 2|E|+2|E_x| + 4 \cdot \smashoperator{\sum_{c \in \calC - (\calC_4 \cup \calT)}}(\|c\|-5)+3\cdot \#\text{\T-cells} - 7|V|\\
        &\leq 2|E|+2|E_x| + 4 \cdot \smashoperator{\sum_{c \in \calC - \calC_4}}(\|c\|-5) - 7|V|\qedhere
    \end{align*}
\end{proof}

\begin{theorem}\label{thm:simple-quasiplanar}
    For every $n \geq 4$, every simple $n$-vertex quasiplanar graph $G$ has at most $6.5n - 20$ edges.
\end{theorem}
\begin{proof}
    We do induction on the number $n$ of vertices in $G$.
    In the base case, $n = 4$, since $G$ is simple, we have at most $\binom{4}{2} = 6$ edges, and indeed $6.5 \cdot 4 - 20 = 6$.

    So assume that $G = (V,E)$ with $|V| = n \geq 5$.
    If $G$ contains a vertex $v$ of degree less than $4$, by induction on $G - v$ we have $|E| \leq \deg(v) + 6.5(n-1) - 20 < 4 + 6.5n - 6.5 -20 < 6.5n - 20$.
    So assume that the minimum degree of~$G$ is $4$.
    If $G$ contains two vertices $u,v$ whose removal separates $G$ into two subgraphs with vertex sets $V_1 \dot\cup V_2$, then by induction on $G-V_1$ and $G-V_2$ we have $|E| \leq 6.5|V_1 \cup \{u,v\}| - 20 + 6.5|V_2 \cup \{u,v\}| - 20 < 6.5n-20$.
    So assume that $G$ is $3$-connected.
    Let $\Gamma$ be a simple quasiplanar drawing of $G = (V,E)$.
    By \cref{lem:simple-quasiplanar-filled} we may assume that $\Gamma$ is filled and hence \cref{lem:simple-quasiplanar} applies.
    Recalling $|\calC_3|=0$, the Density Formula with $t=5$ (\cref{cor:density-formula-5}) gives
    \begin{align*}
        |E|&=5|V|-10 + |\calC_4| - |\calX| -\sum_{c \in \calC_{\geq 5}} (\|c\|-5)\\
        &\leq 5|V|-10+\frac12\left(|\calS_{\rm in}|+\#\text{\A-cells}\right)-|\calX|- \smashoperator{\sum_{c \in \calC_{\geq 5}}} (\|c\|-5) \hspace{4.7em} \text{(by \cref{lem:B-general})}\\
        &\leq 5|V|-10+\frac12\left(\#\text{\A-cells}-|E_x|\right) - \smashoperator{\sum_{c \in \calC_{\geq 5}}} (\|c\|-5) \hspace{5.4em} \text{(by \cref{obs:inner-segments})}\\
        &\leq 5|V|-10+\frac12\left(|E|-3.5|V| +2\cdot \smashoperator{\sum_{c \in \calC_{\geq 5}}} (\|c\|-5)\right) - \smashoperator{\sum_{c \in \calC_{\geq 5}}} (\|c\|-5)\hspace{1.4em} \text{(by \cref{lem:simple-quasiplanar})}\\
        &= 0.5|E| + 3.25|V| - 10,
    \end{align*}
    which implies the desired $|E| \leq 6.5|V|-20$.
\end{proof}

\subsection{\texorpdfstring{$\boldsymbol{k^+}$}{k+}-Real Face Graphs}
\label{sec:k-real-face-graphs}

For an integer $k \geq 1$, a drawing $\Gamma$ of some graph $G$ on the sphere $\mathbb{S}^2$ is \emph{$k^+$-real face} if every cell of $\Gamma$ has at least $k$ vertex-incidences (counting vertices with repetitions), and in this case $G$ is called a \emph{$k^+$-real face graph}.
The $k^+$-real face drawings were recently introduced in~\cite{DBLP:conf/wg/BinucciBDHKLMT23}.
For every $n \geq 3$, simple $n$-vertex $1^+$-real face graphs have at most $5n-10$ edges, simple $n$-vertex $2^+$-real face graphs have at most $4n-8$ edges, and simple $n$-vertex $k^+$-real face graphs for $k \geq 3$ have at most $\frac{k}{(k-2)}(n-2)$ edges
and 
all of these bounds are known to be best-possible~\cite{DBLP:conf/wg/BinucciBDHKLMT23}.

In this section, we reprove all these upper bound results for the connected case by immediate applications of the Density Formula.
In fact, we slightly generalize the results from simple drawings~\cite{DBLP:conf/wg/BinucciBDHKLMT23} to non-homotopic drawings for $k \in \{1,2\}$
and to general (not necessarily non-homotopic) drawings for $k \geq 3$.

\begin{theorem}\label{thm:2-vertex-per-cell}
    For every $n \geq 3$, every connected $n$-vertex non-homotopic $2^+$-real face graph $G$ has at most $4n-8$ edges.
\end{theorem}
\begin{proof}
    In a $2^+$-real face drawing $\Gamma$ of $G$ there clearly are no \C-cells, no \B-cells, no \F-cells, no \A-cells, and no \E-cells.
    As $G$ is connected, so is $\Gamma$, and the claim follows immediately from \cref{lem:4n-8-more-general}.
\end{proof}

\begin{theorem}\label{thm:1-vertex-per-cell}
    For every $n \geq 3$, every connected $n$-vertex non-homotopic $1^+$-real face graph $G$ has at most $5n-10$ edges.
\end{theorem}
\begin{proof}
    Let $\Gamma$ be a $1^+$-real face drawing of $G = (V,E)$.
    As $G$ is connected, so is $\Gamma$.
    As each cell $c$ in $\Gamma$ has at least one incident vertex, there are no \B-cells and no \C-cells, and hence $|\calC_3| = 0$ and $|\calC_4| = \#$\A-cells (cf. \cref{obs:types-of-cells}).
    Therefore, the Density Formula with $t = 5$ (\cref{cor:density-formula-5}) immediately implies
    \[
        |E| \leq 5|V|-10 + 2|\calC_3| + |\calC_4| - |\calX| = 5|V|-10 +\#\text{\A-cells} - |\calX| \leq 5|V|-10,
    \] 
    where the last inequality uses \cref{lem:A-cells-vs-X}.
\end{proof}

\begin{theorem}\label{thm:k-vertex-per-cell}
    For every $k\geq 3, n\geq 3$, every connected $n$-vertex (not necessarily non-homotopic) $k^+$-real face graph $G$ has at most $\frac{k}{k-2}(n-2)$ edges.
\end{theorem}
\begin{proof}
    Let $\Gamma$ be a $k^+$-real face drawing of $G = (V,E)$.
    As $G$ is connected, so is $\Gamma$.
    As each cell $c$ in $\Gamma$ has at least $k$ vertex-incidences, we have $\|c\| \geq 2k$.
    Therefore, using $k \geq 3$, the Density Formula (\cref{lem:density-formula}) with $t = \frac{2k}{2k-4}$ and, thus, $\frac{t-1}{4} = \frac{1}{2k-4}$ immediately gives
    \[
        |E|= \frac{2k}{2k-4}(|V|-2)- \sum_{c \in \calC} \left(\frac{\|c\|}{2k-4}-\frac{2k}{2k-4}\right) - |\calX| \leq \frac{k}{k-2}(|V|-2).\qedhere
    \]
\end{proof}

\subsection{1-Planar and 2-Planar Graphs}
\label{sec:k-planar}

For an integer $k \geq 0$, a drawing $\Gamma$ of some graph $G$ on the sphere $\mathbb{S}^2$ is \emph{$k$-planar} if every edge of $G$ has at most $k$ crossings in $\Gamma$, and in this case $G$ is called a \emph{$k$-planar graph}.
The $k$-planar graphs were introduced by Pach and T\'oth~\cite{PT97-k-planar}, while $1$-planar graphs date back to Ringel~\cite{Rin65-1-planar}.
It is known that for every $n\ge 3$ simple $n$-vertex $0$-planar graphs (these are just planar graphs) have at most $3n-6$ edges~[folklore], simple $n$-vertex $1$-planar graphs have at most $4n-8$ edges~\cite{PT97-k-planar}, and simple $n$-vertex $2$-planar graphs have at most $5n-10$ edges~\cite{PT97-k-planar}.
All of these bounds are known to be best-possible~\cite{PT97-k-planar}.
For arbitrary $k \geq 2$, the best known upper bound is $3.81\sqrt{k}n$~\cite{Ack19-k-planar}, which is tight up to the multiplicative constant.

In this section, we reprove the upper bounds for connected $1$-planar and $2$-planar graphs by immediate applications of the Density Formula.
In fact, we slightly generalize the results from simple graphs as in the original paper~\cite{PT97-k-planar} to connected non-homotopic graphs.

\begin{theorem}\label{thm:1-planar}
    For every $n \geq 3$, every connected non-homotopic $n$-vertex $1$-planar graph $G$ has at most $4n-8$ edges.
\end{theorem}
\begin{proof}
    Any \C-cell, \B-cell, \F-cell, \A-cell, and \E-cell requires an edge with at least two crossings.
    Thus the theorem follows immediately from \cref{lem:4n-8-more-general}.
\end{proof}

\begin{theorem}\label{thm:2-planar}
    For every $n \geq 3$, every connected non-homotopic $n$-vertex $2$-planar graph $G$ has at most $5n-10$ edges.
\end{theorem}
\begin{proof}
    Let $\Gamma$ be a non-homotopic $2$-planar drawing of $G = (V,E)$.
    As an edge with~$k$ crossings gives exactly~$k-1$ inner edge-segments, we have~$|\calS_{\rm in}| \leq |E_x|$ and hence with $|\calS_{\rm in}| = 2|\calX|-|E_x|$ (by \cref{obs:inner-segments}) it follows that $|\calX| \geq |\calS_{\rm in}|$.
    
    As $G$ is connected, so is $\Gamma$.
    By \cref{obs:types-of-cells}, $\calC_3$ is the set of all \C-cells and $\calC_4$ is union of all \B-cells and all \A-cells.
    Therefore, the Density Formula with $t = 5$ (\cref{cor:density-formula-5}) immediately implies
    \begin{multline*}
        |E| \leq 5|V|-10 + 2|\calC_3| + |\calC_4| - |\calX| \leq 5|V|-10 + 2|\calC_3| + |\calC_4| - |\calS_{\rm in}| \\
        \leq 5|V|-10 + 3\cdot\#\text{\C-cells} + 2\cdot\#\text{\B-cells} + \#\text{\A-cells} - |\calS_{\rm in}| \leq 5|V|-10,
    \end{multline*}
    where the last inequality uses \cref{lem:B-general}.
\end{proof}

\end{document}